\documentclass{amsart}

\usepackage{amsthm,lineno}
\usepackage[markup=underlined]{changes}
\usepackage{amssymb}
\usepackage{verbatim}
\usepackage{amsmath}
\usepackage[colorlinks=true, pdfstartview=FitV, linkcolor=blue, 
            citecolor=red, urlcolor=blue]{hyperref}

\newtheorem{theorem}{Theorem}[subsection]
\newtheorem{proposition}[theorem]{Proposition}
\newtheorem{corollary}[theorem]{Corollary}
\newtheorem{lemma}[theorem]{Lemma}
\theoremstyle{definition}
\newtheorem{definition}[theorem]{Definition}
\theoremstyle{remark}
\newtheorem{remark}[theorem]{Remark}

\newcommand\be{\begin{equation}}
\newcommand\ee{\end{equation}}

\numberwithin{equation}{section}
\definechangesauthor[name=CK, color=orange]{CK}
\definechangesauthor[name=TA, color=magenta]{TA}
\definechangesauthor[name=CK, color=green]{JF}
\usepackage{todonotes}
\setremarkmarkup{\todo[color=Changes@Color#1!20,size=\scriptsize]{#1: #2}}

\def\Q{\ensuremath {\mathbb{Q}}}
\def\C{\ensuremath {\mathbb{ C}}}
\def\Z{\ensuremath {\mathbb{Z}}}
\def\R{\ensuremath {{\mathbb{R}}}}
\def\f{\ensuremath {{\mathfrak f}}}
\def\a{\ensuremath {{\mathfrak A}}}

\def\E{\ensuremath {{\mathcal E}}}
\def\O{\ensuremath {{\mathcal O}}}
\def\b{\ensuremath {{\mathfrak b}}}

\title[Eisenstein cocycles over imaginary quadratic fields I]{Eisenstein cocycles over imaginary quadratic fields and special Values of $L$-functions}

\author{Jorge Fl\'orez}
\address{Department of Mathematics, Borough of Manhattan Community College, City University of New York, 199 Chambers Street, New York, NY 10007, USA}
\email{jflorez@bmcc.cuny.edu}

\author{Cihan Karabulut}
\address{Department of Mathematics, William Paterson University, New Jersey 07470, USA} \email{karabulutc@wpunj.edu}

\author{Tian An Wong}
\address{Department of Mathematics and Statistics, Smith College, 10 Elm Street, Northampton, MA 01063, USA}
\email{twong33@smith.edu}



\date{\today}

\begin{document}

\begin{abstract}
We generalize Sczech's Eisenstein cocycle for $GL(n)$ over totally real extensions of $\Q$ to finite extensions of imaginary quadratic fields. By evaluating the  cocycle on certain cycles, we parametrize complex values of Hecke $L$-functions previously considered by Colmez, giving a cohomological interpretation of his algebraicity result on special values of the $L$-functions.
\end{abstract}

\maketitle

\tableofcontents

\section{Introduction}

\subsection{Eisenstein cocycles over totally real fields}

Let $F$ be a totally real number field of degree $n$, and let $\f$ be an integral ideal of $F$. For a fractional ideal $\b$ of $F$ relatively prime to the conductor $\f$, consider the partial zeta function 

\be \label{SC-partial}
\zeta_\f(s,\mathfrak b)=\sum_{\mathfrak a\sim\mathfrak b}\frac{1}{N_{F/\Q}(\mathfrak a)^s}, \quad \text{Re}(s)>1
\ee
where  the sum runs over integral ideals $\mathfrak a \subset F$ equivalent to $\b$ in the narrow ray class group of $F$ mod $\f$. This function has an analytic continuation to the whole complex plane except for a simple pole at $s=1$. Moreover, a classical result of Klingen and Siegel states that $\zeta_\f(s,\mathfrak b)$ assumes rational values at negative integers. 

In \cite{Scz}, Sczech gives another proof of the Klingen-Siegel rationality theorem by showing that the values of $\zeta_\f(s,\mathfrak b)$  at negative integers admit a cohomological interpretation in terms of the group cohomology of $GL_n(\Q)$ as follows: He constructs a cocycle $\Psi$ on $GL_n(\Q)$, called the Eisenstein cocycle, which represents a nontrivial cohomology class in $H^{n-1}(GL_n(\Q),M_{\Q})$ with values in a space of $\Q$-valued distributions denoted by $M_{\Q}$. He then shows that the Eisenstein cocycle $\Psi$ can be evaluated on certain cycles constructed using totally positive units in $1+\f$ to obtain the values of $\zeta_\f(s,\mathfrak b)$ at $s=-1,-2,\dots,$ thereby giving a different proof of the Klingen-Siegel theorem on the rationality of these values.

In recent years, integral versions of the Eisenstein cocycle have been studied by Charollois, Dasgupta, Spiess and others in \cite{CD}, \cite{CDG}, \cite{DM}. For example, in \cite{CD} the authors refine Sczech's Eisenstein cocycle to a certain $l$-smoothed cocycle, where $l$ is a fixed rational prime, and as a result prove certain integrality properties of values of a smoothed partial zeta function. They then interpolate these values to a $p$-adic zeta function, with $p$ prime to $l$, to reprove the celebrated results of Deligne and Ribet \cite{DR}, Cassou-Nogu\`es \cite{CN}, and Barsky \cite{Bar} about the existence of $p-$adic interpolation of the partial zeta functions of totally real fields. These integrality results further allow for a new construction of the  Deligne-Ribet-Cassou-Nogu\`es $p-$adic zeta functions, which play an important role in the recent proof of the Gross-Stark conjecture by Dasgupta, Kakde and Ventullo in \cite{DKV}.

The question of whether these results can be generalized to other partial Hecke $L$-functions associated to other number fields is therefore of great interest. In this paper, we propose such a generalization to the case where $F$ is a degree $n\geq2$ extension of an imaginary quadratic field $K$. As a first step in this generalization and as a main goal of this paper, we generalize Sczech's construction of the Eisenstein cocycle $\Psi$ to define a cocycle $\Psi_s$ on a subgroup $\Gamma$ of $GL_n(K)$ which represents a nontrivial cohomology class in $H^{n-1}(\Gamma, S)$ with values in a function space $S$ that parametrizes the values of the partial Hecke $L$-functions associated to $F$. Unlike the Klingen-Siegel theorem, the special values in our case are only expected to be algebraic up to a transcendental factor. Indeed, special cases of the algebraicity of these values  were studied by Colmez \cite{Col}, and as such this paper can be viewed as a cohomological interpretation of Colmez's result.

In a subsequent paper \cite{FKW}, we study the integrality properties of $\Psi_s$ and using a similar $l$-smoothing operation as in \cite{CD} we obtain an $l$-smoothed cocycle which can be used to deduce the integrality properties of the special values. This in turn will lead to a $p$-adic $L$-function, related to that obtained by Colmez and Schneps \cite{CS}. (For $n=1$, this is the $p$-adic $L$-function of Manin-Vishik \cite{VM} and Katz \cite{K}.) As before, this analytic method may then be used to study properties of this $p$-adic $L$-function. 

As was already mentioned in \cite{Col}, Harder obtains much more general algebraicity results for the special values of $L$-functions attached to algebraic Hecke characters using his theory of Eisenstein cohomology \cite{Ha1},\cite{Ha2}. These results provide a further step in the proof of Deligne's conjecture about the algebraicity of special values of Hecke $L$-functions attached to algebraic Hecke characters (cf.\cite{HS}). Harder constructs his Eisenstein cohomology classes in a very general setting, using very sophisticated methods, such as the theory of representations of semi-simple Lie groups, their cohomology, and the theory of automorphic forms; in particular, the theory of Eisenstein series and their residues, in the sense of Langlands, play a crucial role in the construction of these cohomology classes.

 On the other hand, we define the Eisenstein cocycle, which gives rise to nontrivial cohomology classes in a very explicit way, using analytic methods. The arithmetic applications of the Eisenstein cocycle with regards to the algebraicity of the special values when compared with Harder's Eisenstein cohomology is somewhat limited due to its explicit nature. However, the recent use of the Eisenstein cocycle in proving Iwasawa theoretic results in the totally real setting by Charollois, Dasgupta, Spiess and others suggests that generalizing the Eisenstein cocycle in our setting should also lead to similarly interesting results.

The relationship between the nontrivial cohomology classes corresponding to the Eisenstein cocycle we construct here and the Eisenstein cohomology classes of Harder is not clear, though it is generally believed that the Eisenstein cohomology classes of Harder coincide with the cohomology classes corresponding to the Eisenstein cocycle constructed here. In fact, in the case $n=2$, the results of Weselmann \cite{Wesel} and Obaisi \cite{Oba} show that the cohomology classes corresponding to the Eisenstein cocycle constructed here coincide with the Eisenstein cohomology classes of Harder. 

Finally, we conclude the introduction by stating our main result.

\subsection{Main result} Let $F$ be an extension of an imaginary quadratic field $K$ of degree $n \geq 2$. Fix an   integral ideal $\f$  of $F$ and denote by $I(\f)$ the group of fractional ideals of $F$ prime to $\f$, and by $P(\f)$ the  group of principal ideals with a generator equivalent to $1$ mod $\f$. Moreover, we denote by $G_{\f}:=I(\f)/P(\f)$ the ray class group of $F$ mod $\f$. We also introduce the group of units $U_\f$ of $F$ which are equivalent to 1 mod $\f$, and we let $V_{\f}$ be the free part of the subgroup of units of $U_{\f}$ having relative norm 1 in the extension $F/K$.

Let $\chi:I(\f)\to \C^\times$ be a Hecke character of type $A_0$, i.e.,
\begin{equation}\label{Chi}
\chi\big((a)\big)=\lambda(a),
\end{equation}
for principal ideals  $(a)$ such that $a\equiv 1\, \text{mod}\, \f$, where 
  $\lambda:F^\times\to K^\times$ is a character defined by
\begin{equation}\label{Lambda}
\lambda(a)=\overline{N_{F/K}(a)}^kN_{F/K}(a)^{-l}.
\end{equation}
Here $k\geq 0$  and $l>0$ are both integers 
such that  $\lambda(\epsilon)=1$ for all units $\epsilon$ in the group $U_{\f}$. For the character $\chi$ we consider the $L$-function
\be
\label{Main-L-function}
 L(s,\chi)=\sum_{(\mathfrak{a}, \f)=1}\frac{\chi(\mathfrak{a})}{N_{F/\Q}(\mathfrak{a})^s},
\ee
 where the sum runs over all integral ideals $\mathfrak{a}$ prime to $\f$, and converges for Re$(s)$ large enough.

For a fix integral ideal $\mathfrak b$ of $F$ prime to $\f$, we also consider the partial Hecke $L$-function
\be\label{Partial-Main-L-function}
 \mathcal{L}_{\f}(\chi ,\mathfrak{b};s)=\sum_{ \mathfrak{a} \sim_{\f} \mathfrak{b} }
 \frac{\chi(\mathfrak{a})}{N_{F/\Q}(\mathfrak{a})^s},
\ee
where the sum runs over all integral ideals $\mathfrak{a}$ equivalent to $\mathfrak{b}$ in the ray class group $G_{\f}$, i.e., $\mathfrak{a}\mathfrak{b}^{-1}=(a)$ for some $a\in \mathfrak{b}^{-1}$ such that $a\equiv 1$ mod $\f$. The Hecke $L$-function can be expressed in terms of the partial Hecke $L$-functions as 
\be\label{Main-Partial}
 L(s,\chi)=\sum_{\mathfrak{b}} \mathcal{L}_{\f}(\chi ,\mathfrak{b};s),
\ee
where  $\b$  is running over  a set  of integral representatives of the ray class group $G_\f$.

Our main result can now be summarized as follows (c.f. Corollary \ref{parametrization} below for the precise formulation):

\begin{theorem}\label{mainthm-intro}
Let $\mathcal{L}_{\f}(\chi ,\mathfrak{b};s)$ be the partial Hecke $L$-function defined in Equation \eqref{Partial-Main-L-function}. Then there exists a cocycle $\Psi_s$ in the group cohomology of $\Gamma$ where $n\geq 2$ and a cycle $\E$ in the group homology of $V_\f$ such that 

\begin{equation}
\mathcal{L}_{\f}(\chi ,\mathfrak{b};s)=\frac{\chi(\b)}{N_{F/\Q}(\b)^s}\, C_{F/K,\chi}\,\Psi_s(\E)
\end{equation}
for $\mathrm{Re}(s)>1+\frac{k}{2}$, where  $C_{F/K,\chi}$ is a constant depending on the field extension $F/K$ and the character $\chi$. 
\end{theorem}

\begin{remark}
	The case $n=2$, that is, when $F$ is a quadratic extension of $K$ was first considered in the unpublished thesis of Obaisi \cite{Oba}, whose work we summarize in Section \ref{sec2} below. Most importantly, in this case $\Psi_s$ parametrizes all values of $L(s,\chi)$, whereas for  $n>2$ we are restricted to the region of convergence above. This is in contrast to the result of Sczech for totally real fields, whose Eisenstein cocycle parametrizes the values of partial zeta function $\zeta_\f(s,\mathfrak b)$  defined in \eqref{SC-partial}  only at \textit{integer} arguments, whereas our construction holds for \textit{complex} arguments. This added flexibility is due to the introduction of a convergence factor which is described in equation \eqref{Convergence-factor}.
\end{remark}

Although the setting that we work in is the precise analogue of the totally real extensions of degree $n$ considered by Sczech, in the sense that the rank of the unit group of $F$ is again $n-1$, and these units of infinite order play a key role in the parametrization of the $L$-function, it should be emphasized that Sczech's work over totally real fields does not generalize immediately to our setting due to the following problems: Firstly, the convergence of the cocycle is no longer obtained using the so-called $Q$-summation trick of Sczech, but instead we introduce a convergence factor inspired by Colmez's work. This greatly simplifies our exposition as a major part of \cite{Scz} is devoted to this $Q$-summation. Secondly, for extensions of an imaginary quadratic field one can no longer choose a system of totally positive units, a key condition in the parametrization in \cite[Lemma 6]{Scz}. We prove the analogue of this in Lemma \ref{det} below, which represents the technical heart of this paper. As discussed in \cite[p.597]{Scz}, a similar parametrization is obtained in \cite[Chapitre III]{Col} but is `rather complicated' as it does not use the cocycle property. In fact, our proof holds for any number field; for this reason it may be interesting to ask how Sczech's construction may apply to general number fields.

This paper is organized as follows: In Section \ref{sec3}, we summarize the work of Obaisi which is the case $n=2$ and proceed to construct the Eisenstein cocycle for $n>2$. In Section \ref{sec4}, we construct the cycle $\E$ on which the cocycle $\Psi_s$ is to be evaluated and prove the Theorem \ref{mainthm-intro}. Section \ref{proof-lemma} finally proves the key Lemma \ref{det} required for the parametrization.

\section{The Eisenstein cocycle {$\Psi_s$}}

\label{sec3}

\subsection{The case {$n=2$}}

\label{sec2}

In this section we describe the case $n=2$, in preparation for the general case. We recall that this case was studied in the unpublished thesis of Obaisi \cite{Oba}. We note that the cocycle described here is an inhomegeneous cocycle, whereas in the general case it is more convenient to work with a homogeneous cocycle, as is done in \cite{Scz}.

\subsubsection{The rational cocycle}

Let $\sigma_1,\sigma_2$ be two nonzero column vectors in $\C^2$ and $\sigma=(\sigma_1,\sigma_2)$, and $x$ a row vector in $\C^2$. Now consider the function,
\be
f(\sigma_1,\sigma_2)(x)=\frac{\det(\sigma_1,\sigma_2)}{\langle x,\sigma_1\rangle\langle x,\sigma_2\rangle}
\ee
defined outside the hyperplanes $\langle x,\sigma_1\rangle=0$ and $\langle x,\sigma_2\rangle=0$. For any $A$ in $GL_2(\C)$, we see that $f(A\sigma)(x)=\det(A)f(\sigma)(xA).$

Let $H\subset \C[x_1,x_2]$ be the space of homogeneous polynomials of degree $g$. Then for any $P$ in $H$, we extend $f$ to $H\times \C^2\backslash\{0\}$ as follows:
\be
f(\sigma)(P,x)=P(-\partial_{x_1},-\partial_{x_2})f(\sigma)(x),
\ee
where $\partial_{x_i}$ denotes the partial derivative with respect to $x_i$. 
 We define the action of $A\in GL_2(\C)$ on $P$ by $AP(x)=P(xA)$. This allows us to define an action of $A\in GL_2(\C)$  on $f$ by
  $Af(\sigma)(P,x):=\det(A)f(\sigma)(A^TP,xA)$. 
  This action is homogeneous, that is,  
\be
f(A\sigma)(P,x)=Af(\sigma)(P,x).
\ee

Now define a left 1-cocycle $\psi$ with values in the space of complex-valued functions on $ H\times \C^2\backslash\{0\}$ by
\be
\psi(A)(P,x)=f(e_1,Ae_1)(P,x)
\ee
where $e_1=(1,0)^T$ and $x$ is nonzero. Note that this definition holds only if the denominators  $\langle x,\sigma_1\rangle,\langle x,\sigma_2\rangle$ are nonzero, otherwise one introduces an alternative definition so that a well-defined expression is obtained. See \cite[p.371]{Scz2} for the case of $GL_2(\Q)$ (see \cite[p.16]{Oba}).

\subsubsection{The Eisenstein cocycle}

Fix a quadratic extension $F$ of $K$. We want to sum the 1- cocycle $\psi$ to produce the desired 1-cocycle $\Psi_s$. Let $\Lambda_1,\Lambda_2$ be lattices in $K$ with the same multiplier ring $\O_K\neq \Z$.  Let $R_2$ be the set of matrices $M$ in $GL_2(F)$ whose columns are conjugate over $K$. Let $S'$ be the space of complex-valued function $f$ on $H\times K^2/(\Lambda_1\times\Lambda_2)\times R_2$. We let $\Gamma$ be the following subgroup of $GL_2(K)$ 

$$\Gamma=\left\{\, A \in  \begin{pmatrix}
\Lambda_1\Lambda_1^{-1} &\Lambda_2\Lambda_1^{-1} \\ \Lambda_1\Lambda_2^{-1}&\Lambda_2\Lambda_2^{-1}
\end{pmatrix} \, \bigg| \, \det(A)\in \O_{K}^{\times} \,\right\}\,.$$

Then $\Gamma$ acts on $S'$ by
\be
Af(P,{\bf u},M)=\det(A)f(A^TP,{\bf u}\,A,\,A^{-1}M).
\ee
Define the Eisenstein cocycle
\be\label{OB}
\Psi_s(A)(P,{\bf u},M)=\sum_{x\in{\bf\Lambda}+{\bf u}}\psi(A)(P,x)\Omega_s^k(x,M)
\ee
where ${\bf u}\in K^2$, ${\bf \Lambda}=\Lambda_1\times\Lambda_2$, $k\geq 0$ an integer, and
\be
\Omega_s^k(x,M)=\prod_{i=1}^2\frac{\overline{xM_i}^k}{|xM_i|^{2s}}
\ee
which, by a result of Colmez \cite{Col}, converges absolutely for Re$(s)\gg k$, and in fact has analytic continuation as a function of $s$. It can then be shown that $\Psi_s$ is a 1-cocycle for all $s$, thus represents a cohomology class in $H^1(\Gamma,S')$, c.f. Theorem \ref{cohomologyclass}.

\subsubsection{Parametrizing the partial $L$-function} Obaisi shows that the values of the partial Hecke $L$-function $\mathcal{L}_{\f}(\chi ,\mathfrak{b};s)$ associated to  the quadratic extension $F$ of $K$ can be parametrized by evaluating the Eisenstein cocycle, defined in (\ref{OB}), for specific choices of $A,P,{\bf u},\text{\ and\ }M$. We shall now describe these specific parameters on which the Eisenstein cocycle is evaluated to achieve the parametrization. 

We fix an ideal $\mathfrak{b}$ of $F$ prime to $\f$. By \cite[Theorem 81.3]{Ome}, for the ideal $\mathfrak{fb}^{-1}$ there exists a basis $m_1, m_2$ of the extension $F/K$, as well as integral  ideals $\Lambda_1, \Lambda_2$ of $\mathcal{O}_K$, such that
\begin{equation}
\mathfrak{fb}^{-1}=\Lambda_1m_1+\Lambda_2m_2.
\end{equation}

Let ${\bf u}=(u_1,u_2)\in K^2$ such that $1=m_1u_1+m_2u_2$. Let $V_\f$ be a cyclic subgroup of finite index in $U_\f$  with generator $\epsilon$ such that $N_{F/K}(\epsilon)=1$. The matrices $M, A$ are defined as  $$M=(M_1,M_2) \quad \text{and}\quad \quad A=M\begin{pmatrix}
\epsilon &0\\0&\epsilon'
\end{pmatrix}M^{-1},$$ where $M_1:=(m_1, m_2)^T$, $M_2$ is the conjugate of $M_1$ in $F/K$ and $\epsilon'$ is the conjugate of $\epsilon$ in $F/K$. Finally, the polynomial $P$ is defined to be $$P(x_1,x_2)=(xM_1^{-T})(xM_2^{-T})$$
 where $x=(x_1,x_2)\in \C^2$ and $M_i^{-T}$ denotes the $i$th column of the matrix $M^{-T}$. The following identity, which is Theorem 9 of {\cite{Oba}}, gives the aforementioned parametrization
\be\label{OB-parametrize}
-[U_\f:V_\f]((l-1)!)^2\det(M) \mathcal{L}_{\f}(\chi,\b;s)=\frac{\chi(\b)}{N_{F/\Q}(\b)^s}\,\Psi_s(A)(P^{l-1},{\bf u},M)
\ee
for all $s\in\C$, where $l>0$ is an integer. In view of equation (\ref{Main-Partial}), the identity (\ref{OB-parametrize}) implies that the Eisenstein  cocycle also parametrizes the values of $L(s,\chi)$.

\subsubsection{Algebraicity} The parametrization given in equation (\ref{OB-parametrize}) can be used to obtain an algebraicity result for the special value of Hecke $L$-functions at $s=0$. Obaisi shows that at $s=0$, the Eisenstein cocycle can be expressed as a finite linear combination of products of Eisenstein-Kronecker series, whose algebraicity properties are known by Damerell \cite{D} (see also \cite[Ch.VIII \S15]{W} and \cite[Chap II.0]{Col}). This, together with the identity (\ref{OB-parametrize}) and the equation (\ref{Main-Partial}), immediately implies that the special value $L(0,\chi)$ is algebraic up to a transcendental factor, namely 
\[
L(0,\chi)\in \Omega_{\infty}^{2(k+l)}\,\pi^{-2k}\overline{\mathbb{Q}},
\]
where $\Omega_{\infty}$ is the real period of an elliptic curve defined over $\overline{\mathbb{Q}}$ having complex multiplication by $K$.

\begin{remark}
We would like to point out that Sczech also uses a similar approach in \cite{Scz} to prove the rationality of special values of the partial zeta function $\zeta_\f(s,\mathfrak b)$. This is done by expressing the Eisenstein cocycle in terms of a generalization of the classical Dedekind sums, which are shown to be rational up to a transcendental factor, thus recovering the classical rationality result  of Klingen and Siegel. 
\end{remark}

\subsection{The rational  cocycle}
Having sketched the case $n=2$, we are now ready for the general setting. Throughout we will let $K$ be a fixed imaginary quadratic extension of $\Q$. We first construct an ($n-1$)-cocycle of $GL_n(\C)$ acting on a certain space of functions defined below, following the construction of Sczech for $GL_n(\Q)$, which will allow us to construct the Eisenstein cocycle on $GL_n(K)$.

\begin{definition}\label{function-f}Let $\sigma_1,\dots,\sigma_n$ be the columns of a $n\times n$ matrix $\sigma$ in $GL_n(\C)$. Similar to $n=2$ case, we start with the rational function
\be\label{def-f-1}
f(\sigma)(x)=\frac{\det(\sigma_1,\dots,\sigma_n)}{\langle x,\sigma_1\rangle\dots\langle x,\sigma_n\rangle}
\ee
of a row vector $x$ in $\C^n$, and extend $f$ to $H\times \C^n\backslash\{0\}$ 
\begin{equation}\label{Def-of-f}
f(\sigma)(P,x)=P(-\partial_{x_1},\dots,-\partial_{x_n})f(\sigma)(x)
\end{equation}
where $H$ is the space of homogeneous polynomials $P$ of degree $g\geq1$ in $n$ variables over $\C$, and
where the $\partial_{x_j}$ are the partial derivatives with respect to the variable $x_j$. This function is well-defined outside the hyperplanes $\langle x,\sigma_j\rangle=0$, for all $j=1,\dots,n$. When $\langle x,\sigma_j\rangle=0$ we set $f(\sigma)(x)=0$. Notice that $f(\sigma)$ does not change if $\sigma_j$ is replaced by $\lambda\sigma_j$ for any $\lambda\in \C^\times$, so we may view $\sigma_j$ as points in projective space ${\mathbb P}^{n-1}(\C)$.
\end{definition}

We define the action of $A\in GL_n(\C)$ on a polynomial $P(x)\in H$ by
\be\label{Action-Matrix-poly}
AP(x)=P(xA), \quad x=(x_1,\dots,x_n).
\ee

 The following lemma gives some properties of $f$: 
\begin{lemma}
\label{psi}
The function $f(\sigma)(P,x)$ satisfies the following properties:
\begin{enumerate}
\item
Let $\sigma_0,\dots,\sigma_n$ be nonzero column vectors in $\C^n$. Then
\be
\sum_{i=0}^n(-1)^if(\sigma_0,\dots,\hat{\sigma}_i,\dots,\sigma_n)=0.
\ee
\item
Given $A$ in $GL_n(\C)$, we have 
\be
Af(\sigma)(P,x)=f(A\sigma)(P,x),
\ee
where the action is given by $Af(\sigma)(P,x)=\det(A)f(\sigma)(A^TP,xA).$ 
\end{enumerate}
\end{lemma}

\begin{proof}
These properties follow the same argument as in \cite[pp.586-587]{Scz}, observing that the argument is independent of the coefficient field. For (2), we use the fact that
\be
f(A\sigma)(x)=\det(A)f(\sigma)(xA)
\ee
which follows directly from the definition of $f$.
\end{proof}

\begin{remark}
As a corollary to property (\ref{Def-of-f}), our function can be expressed as
\begin{equation}
\label{Pr}
f(\sigma)(P,x)=\det(\sigma)\sum_r P_r(\sigma)\prod^n_{j=1}\frac{r_j!}{\langle x,\sigma_k\rangle^{1+r_j}}
\end{equation}
where $r$ runs over all partitions of deg$(P)=r_1+\dots+r_n$ into nonnegative parts $r_j\geq 0$, and $P_r(\sigma)$ is a homogeneous polynomial in $\sigma_{ij}$ defined by the series expansion
\be\label{Pr-2}
P(y\sigma^T)=\sum_r P_r(\sigma)\prod_{j=1}^n y^{r_j}_j
\ee
with $\sigma^T$ denoting the transpose of the matrix $\sigma$. In the case of a divided power $P(x)=x_1^{(g_1)}\dots x_n^{(g_n)}$ with $x^{(k)}=x^k/k!$, we have
\be\label{Pr-coeff}
P_r(\sigma)=\sum_{} \prod_{i,j=1}^n\sigma_{ij}^{r_{ij}}
\ee
where the sum runs over simultaneous decompositions 
$r_j=r_{1j}+\dots +r_{nj}, r_{ij}\geq 0$, for $j=1,\dots,n$ satisfying
$r_{i1}+\dots +r_{in}=g_i$.
\end{remark}

\begin{definition}(The rational cocycle.)
Next define an $n$-tuple $(A_1,\dots,A_n)$ with $A_i$ in $GL_n(\C)$, and write $A_{ij}$ for the $j$-th column of $A_i$. Given any $k=1,\dots,n$ and nonzero vector $x\in \C^n$, there is a smallest index $j_k$ such that $\langle x,A_{kj_k}\rangle\neq0$.

Consider the space of complex valued functions
\begin{equation}\label{S0}
S_0=\{f: H\times \C^n\to \C\}.
\end{equation}
 An element $A$ in $GL_n(\C)$ acts on $S_0$ as before by 
\begin{equation}\label{Action-of-A-on-f}
Af(\sigma)(P,x)=\det(A)f(\sigma)(A^TP,xA),
\end{equation}
and $AP(X)=P(XA)$ for $X=(X_1,\dots,X_n)$. Now, define the rational cocycle by a map on $n$ copies
\begin{equation}
\label{Rational-cocycle}
\psi:GL_n(\C)\times\dots\times GL_n(\C)\to S_0
\end{equation}
sending
\begin{equation}\label{Def-of-psi}
(A_1,\dots,A_n)\mapsto f(A_{1j_1},\dots,A_{nj_n})(P,x)
\end{equation}
which we will denote for short by $\psi(A_{1},\dots,A_{n})(P,x)=\psi(\a)(P,x)$, where
$\a$ stands for the $n$-tuple $(A_{1},\dots,A_{n})$. If $x=0$, we set $\psi(A_1,\dots,A_n)(P,x)=0.$
\end{definition}
By Lemma \ref{psi}, $\psi$ is a homogenous $(n-1)$-cocycle on $GL_n(\C)$, i.e., 
\begin{equation}\label{homo}
\psi(A\a)=A\psi(\a),
\end{equation}
and
\begin{equation}\label{cocy}
\sum_{i=0}^{n}(-1)^i\psi(A_0,\dots, \hat{A_i},\dots, A_n)=0.
\end{equation}
This $\psi$ is the analogue of the rational cocycle constructed by R. Sczech for $GL_n(\Q)$. We will use this cocycle to construct the Eisenstein cocycle for an imaginary quadratic field.

\begin{theorem}
The map $\psi$ represents a cohomology class in $H^{n-1}(GL_n(\C),S_0).$ 
\end{theorem}

\begin{proof}
This follows immediately from the properties of $\psi$ described in Lemma \ref{psi}. 
\end{proof}

\begin{remark}
(Alternate expression for $\psi$.) We may rewrite $\psi$ as follows: let $d=(d_1,\dots,d_n)$ be an $n$-tuple of integers with $1\leq d_j\leq n$, and consider the space $A(d)\subset \C^n$ generated by all columns $A_{ij}$ with $j< d_i$, and let $A(d)^\perp$ be the orthogonal complement. Then define
\begin{equation}
\label{X(d)}
X(d)=A(d)^\perp\backslash\bigcup^n_{i=1} A^\perp_{id_i},
\end{equation}
so that if $X(d)$ is nonempty, then it is a subspace of $\C^n$ with a finite number of codimension one subspaces removed. Let
\begin{equation}\label{D}
D=D(A_1,\dots,A_n)=\{d:X(d)\neq \varnothing \}.
\end{equation}
Then by construction of $X(d)$, we can associate with $(A_1,\dots,A_n)$ a finite decomposition of $\C^n$ into a \textit{ disjoint} union
\be
\C^n\backslash\{0\}=\bigcup_{d\in D} X(d).
\ee
In terms of this decomposition, the definition of $\psi$ can be restated as
\begin{equation}
\label{psi2}
\psi(A_{1},\dots,A_{n})(P,x)=f(A_{1d_{1}},\dots,A_{nd_n})(P,x)
\end{equation}
for $x$ in $X(d)$. The cardinality of $D$, i.e, the number of nonempty sets $X(d)$ depends in general on the matrices $A_i$, but we simply observe that the cardinality of $D$ is bounded above by the generic case, that is, when
\be
\dim A(d)=\sum_i (d_i-1)
\ee
for all nonempty $X(d)$. In this case every $d\in D$ satisfies $\sum d_i>2n$, hence $D$ is finite.
\end{remark}

\subsection{The Eisenstein cocycle} 

We now use $\psi$ to construct the cocycle $\Psi_s$ that will parametrize the values of Hecke $L$-functions. 

\begin{definition}\label{Eisenstein-cocycle}
(The Eisenstein cocycle.)
Let ${\bf \Lambda}=\Lambda_1\times\dots\times \Lambda_n$ be a product of $n$ lattices in $K$ with the same ring of multipliers $\mathcal{O}_K$ in $K$. Consider the space $S$ of functions
\begin{equation}\label{S}
S=\{f:H\times K^n/{\bf \Lambda}\times R_n\to \C\}
\end{equation}
where $R_n$ is the set of matrices $M$ in $GL_n(\C)$ such that there exists a degree $n$ extension $F/K$  such that the columns
\be
M_1,\dots,M_n
\ee
are conjugate over $K$. That is,
\begin{equation}
\label{M}
M=\begin{pmatrix}\rho_1(m_1)& \dots& \rho_{n}(m_1)\\ 
\vdots  & \ddots & \vdots\\
\rho_1(m_n) & \dots & \rho_{n}(m_n) \end{pmatrix}
\end{equation}
where $m_1,\dots, m_n\in K$ and the $\rho_i$ are distinct embeddings of $F$ into $\C$, fixing $K$. In particular, all the coefficients of $M$ are nonzero. As before we define $\Gamma$ to be the following subgroup of $GL_n(K)$

$$\Gamma=\left\{\, A \in  \begin{pmatrix}
\Lambda_1\Lambda_1^{-1} &\dots &\Lambda_n\Lambda_1^{-1}\\
\vdots  & \ddots & \vdots\\
\Lambda_1\Lambda_n^{-1} &\dots &\Lambda_n\Lambda_n^{-1}
\end{pmatrix} \, \bigg| \, \det(A)\in \O_{K}^{\times} \,\right\}\,.$$

 The action of $A$ in $\Gamma$ on any $f$ in $S$ is given by
\begin{equation}
Af(P,{\bf u},M)=\det(A)f(A^TP,{\bf u}\,A,A^{-1}M).
\end{equation}

Now define the \textit{Eisenstein cocycle} to be a map $\Psi_s$ from $n$ copies of $\Gamma$ to the space of complex valued functions on $S$, 
\be
\Psi_s:\Gamma\times \dots\times \Gamma\to S
\ee 
by 
\begin{equation}
\label{Szcech-cocycle-sum}
\Psi_s(\a)(P,{\bf u},M)=\sum_{x\in {\bf \Lambda}+{\bf u}}\psi(\a)(P,x)\Omega^k_s(x,M),
\end{equation}
where ${\bf u} \in K^n$ and
\begin{equation}\label{Convergence-factor}
\Omega^k_s(x,M)=\prod_{i=1}^n \frac{\overline{xM_i}^k}{|xM_i|^{2s}}.
\end{equation}
The factor (\ref{Convergence-factor}) will be referred to as the \textit{ convergence factor}. As we shall see from Theorem \ref{convergence-theorem} later, this map is well-defined. That is, the right-hand side of \eqref{Szcech-cocycle-sum} converges absolutely for $\mathrm{Re}(s)> 1+\frac{k}{2}$, with $k\geq 0$, independently of $P,{\bf u},$ and $M$.
\end{definition}

\begin{theorem}
\label{cohomologyclass}
Suppose that $s\in\C$ satisfies $\mathrm{Re}(s)> 1+\frac{k}{2}.$ The Eisenstein cocycle $\Psi_s$ represents a nontrivial cohomology class in the group $H^{n-1}(\Gamma,S)$, which we denote by $[\Psi_s]$.
\end{theorem}

\begin{proof}
Let $A_0,\dots,A_n$ be elements in $\Gamma$. Since $\psi$ is a $(n-1)$-cocycle, we know that
\be
\sum_{i=0}^n(-1)^i\psi(A_0,\dots,\hat{A_i},\dots,A_n)(P,x)=0
\ee
for a fixed $P$ and $x$. Multiplying by the convergence factor $\Omega_s^k(x,M)$ and summing over $x$ in ${\bf \Lambda}+\bf u$ we have
\be
\sum_{x\in{\bf \Lambda}+\bf u}\sum_{i=0}^n(-1)^i\psi(A_0,\dots,\hat{A_i},\dots,A_n)(P,x)\,\Omega_s^k(x,M)=0,
\ee
where the sum converges for Re$(s)$ large enough. Then interchanging the sums and using the definition of $\Psi_s$ we have
\be
\sum_{i=0}^n(-1)^i\Psi_s(A_0,\dots,\hat{A_i},\dots,A_n)(P,x)=0,
\ee
therefore $\Psi_s$ is a cohomology class in $H^{n-1}(\Gamma,S)$.

To see that $\Psi_s$ is nontrivial, we use the property that $\Psi_s$ 
parametrizes  the $L$-function $\mathcal{L}_{\f}(\chi,\mathfrak{b};s)$, from equation (\ref{Partial-Main-L-function}),  as follows.
Fix $\f$, $\b$, $P$, ${\bf u}$ and $M$ as in Corollary \ref{parametrization}.
Then pairing $\Psi_s(.)(P,{\bf u},M)\in H^{n-1}(\Gamma, \C)$ with the cycle $\E=\E[\b,M]\in H_{n-1}(\Gamma,\Z)$ (cf. (\ref{the cycle})), via the nondegenerate bilinear pairing
\be
\langle \cdot,\cdot \rangle : H^{n-1}(\Gamma, \C) \times H_{n-1}(\Gamma,\Z)\to H_0(\Gamma,\C) = \C,
\ee
 we have, after  Corollary \ref{parametrization}, that 
 \be
 \langle [\Psi_s(.)(P,{\bf u},M)],[\E] \rangle=\Psi_s(\E)(P,{\bf u},M)
 \ee
 parametrizes   the partial $L$-function $\mathcal{L}_{\f}(\chi,\mathfrak{b};s)$. In particular, since $\mathcal{L}_{\f}(\chi,\mathfrak{b};s)$ is non-trivial, then 
 \be
 \langle [\Psi_s(.)(P,{\bf u},M)],[\E] \rangle\neq 0,
 \ee
 and so the cocycle $\Psi_s$ is non-trivial as well.
\end{proof}

\subsection{Convergence}

We conclude this section by determining a domain of absolute convergence for the cocycle $\Psi_s$, with the aid of the convergence factor. First, we make the following observation:

\begin{lemma}
\label{factor}
Let $x$ be a row vector in $K^n$, $M\in R_n$ and $\sigma\in GL_{n}(\mathcal{O}_K)$. Then
\be
\Omega_s^k(x\sigma,M)=\Omega_s^k(x,\sigma M).
\ee
\end{lemma}

\begin{proof}
Observe that since the entries of $\sigma$ are in $\mathcal{O}_K$, then $\sigma M$ is 
also in $R_n$. Now,
from the associativity of the product of matrices we have $x(\sigma M)=(x\sigma )M$. Therefore by comparing the $i$-th entry on each side we get 
$\langle x\sigma,M_i \rangle=\langle x,(\sigma M)_i \rangle$. The lemma follows.
\end{proof}

\begin{remark}
Using the partition $X(d)$ in (\ref{X(d)}), we may rearrange the sum (\ref{Szcech-cocycle-sum}) as in \cite[p.598]{Scz} to obtain the expression
\be
\Psi_s(\mathfrak{A})=\sum_{d\in D}\det(\sigma)\sum_r P_r(\sigma)G_r(\sigma),
\ee
where
\be
G_r(\sigma)
   =\sum_{x\in X(d)\cap {\bf \Lambda}} 
         \Omega^k_s(x,M)
             \prod_{j=1}^n 
                  \frac{r_j!}{\langle x,\sigma_j\rangle^{1+r_j}}.
\ee
\end{remark}

The following result of Colmez, from Colloraire 1 of \cite[p.196]{Col}, plays a key role in the proof of absolute convergence.  We will consider $\C^n$ endowed with the sup norm, i.e,
\be
||w||=\max_{i=1,\dots,n}\big\{|w_i|\big \}
\ee
for any $w=(w_1,\dots,w_n)$ in $\C^n$.

\begin{lemma}[Colmez]
\label{colmez2}
Let ${\bf \Lambda}$ be a lattice in $\C^n$. Assume that for all $\epsilon>0$ there exists a constant $C(\epsilon)>0$ such that if $w\in{\bf\Lambda}$ and $\prod_{i=1}^n|w_i|\leq C(\epsilon)||w||^{-\epsilon}$ then $w=0$. Then the series
\be
\sum_{\substack{w\in\bf\Lambda\\ w\neq0}}\prod_{i=1}^n|w_i|^{-2\gamma}||w||^{-\mu}
\ee
converges if $\mu>0$ and $\gamma>1$. 
\end{lemma}

\begin{remark}\label{Colmez-Remark}
Observe that the above lemma  also holds for any translate $\bf \Lambda+u$, $\mathbf{u} \in \mathbb{C}^n-\bf{\Lambda}$, of the lattice $\bf \Lambda$. Indeed, this follows from the fact that there exists a positive constant $C$ such that $||w||\leq C||w+\bf u||$ for any $w\neq  0, \bf u$ in $ \bf \Lambda$.
\end{remark}

\begin{theorem}\label{convergence-theorem}
The cocycle $\Psi_s$ converges absolutely for Re$(s)> 1+\frac{k}{2}.$ In particular, this convergence is independent of $P, {\bf u},$ and $M$.
\end{theorem}

\begin{proof}
From the decomposition of $\mathbb{C}^n$ into the sets $X(d)$, we can partition the lattice as disjoint union
\be
{\bf \Lambda}+{\bf u}=\dot{\bigcup_{d\in D}} {\bf \Lambda}_d,
\ee
where 
${\bf \Lambda}_d=({\bf \Lambda}+{\bf u})\cap X(d)$. Thus 
\be
\Psi_s(\a)(P,{\bf u},M)
=\sum_{d\in D}\sum_{x\in {\bf \Lambda}_d}\psi(\a)(P,x)\,\Omega^k_s(x,M).
\ee
Since $D$ is finite, it suffices to show that for a fixed $d\in D$ the sum
\be
\sum_{x\in  {\bf \Lambda}_d}\psi(\a)(P,x)\Omega^k_s(x,M)
\ee
is convergent for Re$(s)> 1+\frac{k}{2}.$ In particular, the range of convergence is uniform in $\a,P, {\bf u},$ and $M$.

Observe that for $x\in {\bf \Lambda}_d$, the term $\psi(\a)(P,x)$ can
be expressed as $f(A_{1d_1},\dots, A_{nd_{n}})(P,x)$.
Let us denote the matrix $(A_{1d_1},\dots, A_{nd_{n}})$ by $\sigma=(\sigma_1,\dots, \sigma_n)$.  From (\ref{Pr}) we have that
\be
\sum_{x\in {\bf \Lambda}_d}\psi(\a)(P,x)\Omega^k_s(x,M)
    =\sum_r \det(\sigma)P_r(\sigma)
       G_{r}(\sigma),
\ee
where
\be
G_r(\sigma)=\sum_{x\in {\bf \Lambda}_d}
    \Omega_s^k(x,M)
    \prod_{j=1}^n 
                  \frac{r_j!}{\langle x,\sigma_j\rangle^{1+r_j}}.
\ee
We may assume that $\det(\sigma)\neq 0$, 
otherwise the above sum is zero. Also, by a change  of variables 
$w=xM$, and letting $M'=M^{-1}\sigma$ and ${\bf \Lambda}_d'={\bf \Lambda}_dM$, we obtain the simplified expression
\be
G_r(\sigma)=\sum_{w\in {\bf \Lambda}_d' }\prod_{j=1}^n\frac{r_j!}{(wM'_j)^{1+r_j}}\frac{\overline{w_j}^k}{|w_j|^{2s}}.
\ee

 Then it suffices to bound $G_r(\sigma)$. We have trivially by the triangle inequality
\be
|G_r(\sigma)|\leq\sum_{w\in {\bf \Lambda}_d'}\prod_{j=1}^n\frac{r_j!}{|wM'_j|^{1+r_j}}\frac{1}{|w_j|^{2\mathrm{Re}(s)-k}}.
\ee
By Lemma \ref{Lattice-bound} below, the constant  
\be
C=\max_{\substack{w\in {\bf \Lambda}_d'\\ i=1,\dots, n}} 
    \left\lbrace  |wM'_{i}|
        \prod_{j=1}^n  \frac{1}{ |wM'_j|^{1+r_j}}  \right\rbrace  
        =\max_{\substack{w\in {\bf \Lambda}_d'}} 
    \left\lbrace  ||wM'||
        \prod_{j=1}^n  \frac{1}{ |wM'_j|^{1+r_j}}  \right\rbrace  
\ee
is finite. Then
\begin{align}\label{G-r}
|G_r(\sigma)|&\leq C \sum_{w\in {\bf \Lambda}_d'}\prod^{n}_{j=1}\frac{r_j!}{|w_j|^{2\text{Re}(s)-k}}\frac{1}{||wM'||}\notag\\
&\leq ||M'||\,C\sum_{w\in {\bf \Lambda}_d'}\prod^{n}_{j=1}\frac{r_j!}{|w_j|^{2\gamma}}\frac{1}{||w||}\\
&\leq ||M'||\,C\sum_{w\in ({\bf \Lambda+u})M}\prod^{n}_{j=1}\frac{r_j!}{|w_j|^{2\gamma}}\frac{1}{||w||},
\end{align}
where $\gamma=\text{Re}(s)-\frac{k}{2}$, and $||M'||$ is a constant such that $||w||\leq ||M'||\cdot ||wM'||$ for all $w$; this constant $||M'||$ exists since the matrix $M'$ is invertible, so as an operator its inverse is bounded. Then by Lemma \ref{colmez2} and Remark \ref{Colmez-Remark}, applied to the smallest lattice containing  $({\bf \Lambda+u})M$, we see that 
the last term in the chain of inequalities above  converges absolutely for $\gamma>1$, which shows that  $G_r(\sigma)$
is bounded.
\end{proof}

\begin{remark}
     The hypothesis of Lemma \ref{colmez2} is satisfied 
        for ${\bf \Lambda}_d$ for the following reason. 
        For a fixed  $1\leq i\leq n$, let $t_i=1+\epsilon$ 
        and $t_j=1$ for $1\leq j\leq n$, $ j\neq i$. 
        Then, by Lemma \ref{Lattice-bound}, there exists 
        a constant $C_i(\epsilon)>0$ such that 
        \be C_i(\epsilon)<\prod_{j=1}^n|w_j|^{t_j} \ee
        for every $w\in {\bf \Lambda}_d$. 
        Let $C(\epsilon)=\min_{1\leq i\leq n}\{C_i(\epsilon)\}$, 
        then \be C(\epsilon)<||w||^{\epsilon}\prod_{j=1}^n|w_j|,\ee 
        for every $w\in {\bf \Lambda}_d$.
        \end{remark}

To conclude the proof, we obtain the constant $C$ as required above.

\begin{lemma}
\label{Lattice-bound}
Let ${\bf \Lambda}$ be a lattice in $\C^n$ and  real numbers $t_i\geq 0$, $i=1,\dots,n$.
Let ${\bf \Lambda'}:=\{w\in {\bf \Lambda}\ :\ \prod_{i=1}^nw_i\neq 0\}$. Then the supremum
\be
\sup_{w\in {\bf \Lambda'}}\left( \prod_{i=1}^n|w_i|^{-t_i}\right)
\ee
is finite.
\end{lemma}

\begin{proof}
Since ${\bf \Lambda}$ is a lattice there exists a constant $C>0$ such that $C<|w_i|$ for all 
$w\in {\bf \Lambda'}$. Then 
\be
C^{t_1+\cdots+t_n}< \prod_{i=1}^n|w_i|^{t_i}  
\ee
for all $w\in {\bf \Lambda'}$ and the result follows.
\end{proof}

\section{Values of $L$-functions}

\label{sec4}

In this section, we introduce the Hecke $L$-functions whose special values we will parametrize using the Eisenstein cocycle. In particular, we will construct a cycle $\E[\b,M]$ on which the Eisenstein  cocycle will be evaluated. The value of the cocycle on $\E[\b,M]$ will then be related to the $L$-function. As the proof of this relation is rather technical, for the sake of exposition we delay it to the Section \ref{proof-lemma}. Indeed, assuming this relation it is straightforward to parametrize the $L$-function, which we do in this section.

\subsection{An expression for the Hecke $L$-function} For the rest of this paper, we fix a degree $n$ extension $F$ of an imaginary quadratic extension $K$.

\begin{definition}\label{The partial-L}(The partial $L$-function.)
Let $\mathfrak{b}$ be an integral ideal of $F$ prime to $\f$.  For  an element $r$ in $\mathfrak{b}^{-1}$, and $\lambda$ as in (\ref{Lambda}), we define the partial $L$-function as
\begin{align}
\label{Partial-L}
\mathcal L(\mathfrak{b},r,s) &=\sum_{\substack{(a)\in I(\f)\\ a\in \f \mathfrak{b}^{-1}+r}}\lambda(a)N_{F/\Q}((a))^{-s}\\
\label{Partial-La}
&=\sum_{a\in \f \mathfrak{b}^{-1}+r/U_\f}\lambda(a)N_{F/\Q}((a))^{-s}
\end{align}
The second equality follows from the observation that for two principal ideals $(a)$ and $(b)$, where $a,b\in \f \mathfrak{b}^{-1}+r,$ the identity $(a)=(b)$ holds if and only if $a=bw$ for some $w$ in $U_\f$.
\end{definition}
The following proposition expresses $L(s,\chi)$ from (\ref{Main-L-function}) in terms of $\mathcal L(\b,r;s)$. 
\begin{proposition}
\label{fullpartial1}
For an integral ideal $\mathfrak{b}$ of $F$ prime to $\f$, the partial  Hecke $L$-function $\mathcal{L}_{\f}(\chi,\mathfrak{b};s)$ from Equation \eqref{Partial-Main-L-function} can be written in terms of the partial $L$-function \eqref{Partial-La} as follows
\begin{equation}
\label{fullypartial}
\mathcal{L}_{\f}(\chi,\mathfrak{b};s)=\frac{\chi(\b)}{N_{F/\Q}(\b)^s}\,
       \mathcal L(\b,1,s).
\end{equation}
In particular, from \eqref{Main-Partial} we have that
\begin{equation}
\label{fullpartial}
L(s,\chi)=\sum_{\b}\frac{\chi(\b)}{N_{F/\Q}(\b)^s}\,
       \mathcal L(\b,1,s),
\end{equation}
where $\b$ is running over a set of integral representatives of $G_{\f}$, the ray class group  of $F$ mod $\f$.
\end{proposition}

\begin{proof}
Fix an integral ideal $\mathfrak{b}$ of $F$ prime to $\f$. We can write
\be\label{inner}
\mathcal{L}_{\f}(\chi,\mathfrak{b};s)=\frac{\chi(\b)}{N_{F/\Q}(\b)^s} \sum_{\mathfrak{a}\sim_{\f} \b}\frac{\chi(\mathfrak{ab}^{-1})}{N_{F/\Q}(\mathfrak{ab}^{-1})^s},
\ee
where the sum is running over all integral ideals $\mathfrak{a}$  equivalent to $\mathfrak{b}$ in the ray class group of $F$ mod $\f$, i.e., $\mathfrak{ab}^{-1}=(a)$ for some $a$ in $\b^{-1}$ such that $a \equiv 1 \text{mod}\ \f $. 

On the other hand, notice that for $r_1,r_2\in \b^{-1}$, $(r_1)=(r_2)$ if and only if $r_1\equiv r_2\pmod{\f\b^{-1}}$. Thus we can write the inner sum of (\ref{inner}) as
\begin{align}
\sum_{\mathfrak{a}\sim_{\f} \b}
     \frac{\chi(\mathfrak{ab}^{-1})}{N_{F/\Q}(\mathfrak{ab}^{-1})^s}
          &=\sum_{\substack{(a)\in I(\f)\\ a\in \b^{-1},\ a\equiv 1\ \text{mod}\, \f}}
               \frac{\lambda(a)}{N_{F/\Q}((a))^s}\notag\\
                    &=
                      \sum_{\substack{(a)\in I(\f)\\ a\in \f \mathfrak{b}^{-1}+1}}          
                           \lambda(a)N_{F/\Q}((a))^{-s}\notag\\
                  &=\mathcal L(\b,1,s)\notag,
\end{align}
and the identity (\ref{fullypartial}) follows.
\end{proof}

\subsection{The cycle $\E[\b,M]$ and the parametrization of the partial $L$-function}\label{the cycle-section}

Let us first introduce the notation we will use  for the rest of the paper. We fix an ideal $\mathfrak{b}$ of $F$ prime to $\f$  and also an element $r$ of $\b^{-1}$. By \cite[Theorem 81.3]{Ome}, for the ideal $\mathfrak{fb}^{-1}$ there exists a basis $m_1,\dots, m_n$ of the extension $F/K$, as well as integral  ideals $\Lambda_1,\dots, \Lambda_n$ of $\mathcal{O}_K$, such that
\begin{equation}\label{fbinverse}
\mathfrak{fb}^{-1}=\Lambda_1m_1+\cdots+\Lambda_nm_n.
\end{equation}
Moreover, by
\cite[81.5]{Ome} we can further assume that $\Lambda_i=\mathcal{O}_K$, for $2\leq i \leq n$. However,  this stronger assumption is not needed for our purposes.

For this given basis, there exists elements $u_1,\dots, u_n\in K$ such that $r=u_1m_1+\cdots+u_nm_n$. Thus
\begin{equation}\label{Fb+r}
\f \mathfrak{b}^{-1}+r=(\Lambda_1+u_1)\,m_1+\cdots+(\Lambda_n+u_n)\,m_n.
\end{equation}
Let $M_i$ denote the column vector $(\rho_i(m_1),\dots, \rho_i(m_n))^T$, $1\leq i\leq n$,
where $\rho_1,\dots,\rho_{n}$ denote the embeddings of $F$ into $\C$, fixing
the imaginary quadratic field $K$. Let
\begin{equation}\label{M-u-Lambda}
M=\big(\,M_1,\dots,M_n\,\big),\quad
{\bf u}=(u_1\dots, u_n),\quad \text{and}\quad
{\bf \Lambda}=\Lambda_1\times\dots\times\Lambda_n.
\end{equation}

Associated to the basis $m_1,\dots, m_n$ there is also
a  representation 
$\varrho: \mathcal{O}_F^{\times}\to GL_n(K)$ defined as 
\begin{equation}\label{varrho}
\varrho(\eta)=M\delta(\eta)M^{-1},
\end{equation}
 where 
$\delta(\eta)$ is the diagonal matrix
\begin{equation}\label{delta-map}
\delta(\eta)=\begin{pmatrix}
\rho_1(\eta)	&		& \\ 
				& \ddots	& \\
				&		& \rho_n(\eta)
\end{pmatrix}.
\end{equation}
The transpose of the matrix (\ref{varrho}) corresponds to the matrix representation of the 
linear transformation from $F$ to $F$ given by multiplication-by-$\eta$ in $F$ with respect to
the basis $m_1,\dots, m_n$ of the extension $F/K$. Thus, it indeed lies 
in $GL_n(K)$. To be more specific, from the fact that 
$\eta\, \f\b^{-1}=\f\b^{-1}$ for any unit $\eta\in \mathcal{O}_F^{\times}$, it follows
after (\ref{fbinverse}) that $\varrho(\eta)$ has entries in
\be
\begin{pmatrix}
\Lambda_1\Lambda_1^{-1} &\dots &\Lambda_n\Lambda_1^{-1}\\
\vdots  & \ddots & \vdots\\
\Lambda_1\Lambda_n^{-1} &\dots &\Lambda_n\Lambda_n^{-1}
\end{pmatrix}.
\ee

Additionally, there are  two norm forms associated to $m_1,\dots,m_n$, namely
\begin{equation}\label{P and Q}
Q(x)=\prod_{i=1}^n x\,M_i\quad \text{and}\quad
P(x)=\prod_{i=1}^n x\,M_i^{-T}, \quad x\in \C^n.
\end{equation}
Here $M_i$ and $M_i^{-T}$ denote the $i$th column of 
the matrices $M$ and $M^{-T}$, respectively, defined  in (\ref{M-u-Lambda}). 
Notice that, for $x=(x_1,\dots, x_n)\in F^n$, we have
\begin{equation}\label{Q-norm}
Q(x)=N_{F/K}(\xi),\quad \xi=\sum_{i=1}^n\, x_i\,m_i.
\end{equation}
Furthermore, a simple calculation
gives the following relationships
\begin{equation}\label{Action-P-Q}
\big(\varrho(\alpha)Q\big)(x)=N_{F/K}(\alpha)Q(x)\quad \text{and} \quad \big(\varrho(\alpha)^{T}P\big)(x)=N_{F/K}(\alpha)P(x),
\end{equation}
for any $\alpha$ in $\mathcal{O}_F^{\times}$ and any $x$ in $\C^n$, where the action of a matrix on a polynomial is defined as in  Equation (\ref{Action-Matrix-poly}).

Finally, we let $V_\f$ be the free part of the group $U_{\f}^{(1)}=\{\epsilon\in U_{\f}\,:\,N_{F/K}(\epsilon)=1\}$, i.e., the subgroup generated by the elements of infinite order in $U_{\f}^{(1)}$. Each one of the groups
\be
V_\f\subset U_{\f}^{(1)}\subset U_\f \subset \mathcal O_F^\times .
\ee
is  of rank $n-1$ since they have finite index in $\mathcal O_F^\times$  and
$
\mathrm{rank}_\Z(\mathcal O_F^\times)=n-1,
$
by the Dirichlet unit theorem.


Thus we will choose a fix set of generators  $\epsilon_1,\dots,\epsilon_{n-1}$ of $V_{\f}$, i.e.,
\begin{equation}\label{norm1}
V_{\f}=\left\langle \epsilon_1\dots, \epsilon_{n-1} \right\rangle
\quad  \mathrm{and} \quad N_{F/K}(\epsilon_i)=1,\  i=1,\dots, n-1.
\end{equation}
With these observations, the sum (\ref{Partial-L}) 
can now be expressed as 
\begin{equation}\label{Partial-L-V}
\mathcal L(\mathfrak{b},r,s)=[U_\f:V_\f]^{-1}\sum_{a\in \f \b^{-1}+r/V_\f}\lambda(a)N_{F/\Q}((a))^{-s}.
\end{equation}

Let  $\mathfrak{V}_{\f}=\varrho(V_\f)$. From (\ref{norm1}), this is a subgroup of $\text{SL}_n(K)$ generated by the matrices
\begin{equation}\label{A_i}
A_i=\varrho(\epsilon_i).\quad i=1,\dots,n-1. 
\end{equation}

\begin{remark}
Note that in most cases  $V_\f$ coincides with
the group  $U_{\f}^{(1)}=\{\epsilon\in U_{\f}\,:\,N_{F/K}(\epsilon)=1\}$.
This is true, for example, when $N_{F/\Q}(\f)$ is large enough. 
Indeed, this last condition forces
$U_{\f}^{(1)}$ to have no roots of unity, thus being free. To see this, just take $\f$
such that $N_{F/\Q}(\f)$ is larger than any of the elements in the finite set
\be
\left\lbrace \big|N_{F/\Q}(\zeta-1)\big|\,:\,\zeta\in (\mathcal{O}_F^{\times})_{\text{tor}} \right\rbrace, 
\ee
where $(\mathcal{O}_F^{\times})_{\text{tor}}$ is the group of roots of unity in $F$, which  is finite.
\end{remark}

\subsubsection{The cycle {$\E[\mathfrak{b},M]$}}
In this section we will construct the cycle $\E=\E[\mathfrak{b},M]$, which depends on $\mathfrak{b}$ and  $M$, with which the Eisentein cocycle will be paired, giving the parametrization of $L$-values. 
\begin{definition}
\label{cycle}
(The cycle)
 Define the chain in $C_{n-1}(\mathfrak{V}_{\f},\Z)$,
\begin{equation}\label{the cycle}
\E=\E[\mathfrak{b},M]=\rho\sum_\pi \text{sign}(\pi)[A_{\pi(1)}|\dots|A_{\pi(n-1)}]
\end{equation}
where 
\be
[A_1|\dots|A_n]=(1,A_1,A_1A_2,\dots,A_1A_2\dots A_n)
\ee
and 
\be
\rho=(-1)^{n-1}\text{sign}(R_{F/K}).
\ee
Here $R_{F/K}=R(\epsilon_1,\dots, \epsilon_{n-1})$ is the regulator $\det(2\log|\rho_i(\epsilon_j)|)$, $1\leq i,j\leq n-1$, and $\pi$ runs over all permutations of the set $\{1,\dots,n-1\}$. The factor $\rho$ is introduced so that $\E[\b, M]$ is independent of the order of
the units $\epsilon_1,\dots, \epsilon_{n-1}$ and the simplices involved  are positively oriented (cf. \cite[p.596]{Scz}).

From Lemma 5 of \cite[p.592]{Scz}, one sees that this chain defines a  cycle in $H_{n-1}({\mathfrak{V}_{\f}},\Z)$,  denoted by $[\E]$, whose  homology class  is independent of the set of generators of $V_\f$. In general, if $V$ is a subgroup of finite index in $V_\f$ with generators
$\epsilon_1',\dots, \epsilon_{n-1}'$ we define
\begin{equation}\label{cycle-V}
\E[\mathfrak{b},M,V]=\rho\sum_\pi \text{sign}(\pi)[A'_{\pi(1)}|\dots|A'_{\pi(n-1)}]
\end{equation}
where $A'_i=\varrho(\epsilon_i')$, $i=1,\dots, n-1$, and $\rho$ in this case is $(-1)^{n-1}\text{sign}(R(\epsilon_1',\dots, \epsilon_{n-1}'))$. This also defines  
a homology  class in $H_{n-1}(\mathfrak{V},\Z)$ independent of the set of generators of $V$; here  $\mathfrak{V}=\varrho(V)$.

\end{definition}

The identity in the next lemma is the key to parametrizing values of the $L$-functions using the Eisenstein cocycle. 
\begin{lemma}\label{lemma6}
\label{det}
 With notation as above we have
\begin{equation}
\label{deteq}
\sum_{A \in {\mathfrak{V}_{\f}}}\psi\big(\,A\, \E[\b,M]\, \big)(P^{l-1},x)=\det(M)\frac{((l-1)!)^n}{Q(x)^l},
\end{equation}
for any nonzero $x\in{\bf \Lambda}+\bf u$ and  any positive integer $l$.
\end{lemma}

\begin{proof}
See  Section \ref{proof-lemma}.
\end{proof}

\begin{remark}
(1) In Lemma \ref{independence}, we shall show that the left-hand side of this identity is independent of the choice of a  subgroup $V$ of finite index in $V_\f$ and any choice of generators for $V$.

(2) The proof of the identity follows the ideas of \cite[Lemma 6]{Scz}, and is related to the proof of the Dirichlet unit theorem. Indeed, the main ingredient of the proof is a certain geometric construction that realizes a cancellation in the sum using the cocycle property of $\psi$. But in our case one no longer has a system of totally positive units, which is crucial to Sczech's proof, as they provide certain estimates on the translates  of the fundamental domain. The estimate is complicated by the possibility that the translates may have nontrivial intersection. (See \S\ref{strategy} for a detailed summary.) In our case, we are able to modify the simplices in question so as to avoid the approximation argument needed for the estimates. 

(3) Sczech also points out in his case \cite[pp.586, 597]{Scz2} that an alternative proof of the Lemma can be given using an identity of Hurwitz from which the left hand side of the equation \eqref{deteq} can be seen as integrating along the translates of a fundamental domain of the group of units.
\end{remark}

\subsubsection{Parametrization of  $\mathcal L(\b,r,s)$}

We are now ready to parametrize the values of the partial $L$-function, which will lead us to the main result.

\begin{theorem}
\label{maintheorem}
Let $\f ,\b$ be relatively prime integral ideals of $F$ and $r \in \b^{-1}$. 
Let $M$, ${\bf u}$ and ${\bf \Lambda}$ as in \eqref{M-u-Lambda} and $P$
as in \eqref{P and Q}. Then
\be
\Psi_s\big(\,\E[\b,M]\,\big)(P^{l-1},{\bf u},M)
    =-\det(M)\,((l-1)!)^n\,[U_\f:V_{\f}]\,\mathcal L(\b,r,s)
\ee
for $\mathrm{Re}(s)>1+\frac{k}{2}$ and  any  positive integer $l$. This identity also holds if we replace $V_\f$
by any of its subgroups of finite index $V$, in which case, $\E[\b,M]$ has to  be replaced
by $\E[\b,M,V]$, and $[U_\f:V_\f]$ by $[U_\f:V]$.
\end{theorem}

\begin{proof}
Let $\E=\E[\b,M]$. We start by expressing the left-hand side of the identity (being in the range of absolute convergence) as 
\begin{align}
\Psi_s(\E)(P^{l-1},{\bf u},M)=&\sum_{x\in{\bf \Lambda+u}}\psi(\E)(P^{l-1},x)\,\Omega_s^k(x,M)\notag\\
=&\sum_{x\in{(\bf \Lambda}+{\bf u})/\mathfrak{V}_{\f}}
       \,\sum_{A\in \mathfrak{V}_{\f}}\psi(\E)(P^{l-1},xA)\,\Omega_s^k(xA,M)\notag\\
=&\sum_{x\in{(\bf \Lambda}+{\bf u})/\mathfrak{V}_{\f}}\Omega_s^k(x,M)\sum_{A\in \mathfrak{V}_{\f}}\psi(\E)(P^{l-1},xA),
\end{align}
where the last equality follows from 
\be
\Omega_s^k(xA,M)=\Omega_s^k(x,M),\quad A\in \mathfrak{V}_{\f}=\varrho(V_{\f}).
\ee
This last identity is in turn a consequence of Lemma \ref{factor} and Equation (\ref{Action-P-Q}) 
together with the fact that 
$N_{F/K}(\epsilon)=1$ for every $\epsilon$ in $V_{\f}$, and  also
$\Omega_s^k(x,M)=\overline{Q(x)}^k/|Q(x)|^{2s}$. 
Next, by the definition of the action of $GL_n$ on $S_0$ 
and the fact that $A^TP^{l-1}=P^{l-1}$ for $A\in \mathfrak{V}_{\f}$, which  can be verified from  (\ref{A_i}) and (\ref{Action-P-Q}), it follows that
\be
\psi(\E)(P^{l-1},xA)=A\psi(\E)(P^{l-1},x).
\ee
Furthermore, since the cocycle is homogeneous, we have
\be
A\psi(\E)(P^{l-1},x)=\psi(A\E)(P^{l-1},x).
\ee
Altogether this gives 
\be
\Psi_s(\E)(P^{l-1},{\bf u},M)
  =\sum_{x\in{\bf \Lambda}+{\bf u}/\mathfrak{V}_{\f}}\Omega_s^k(x,M)
           \sum_{A\in \mathfrak{V}_{\f}}\psi(A\E)(P^{l-1},x).
\ee
Applying Lemma \ref{det}, this expression equals 
\be
\det(M)((l-1)!)^n\sum_{x\in{\bf \Lambda}+{\bf u}/\mathfrak{V}_{\f}}\frac{\Omega_s^k(x,M)}{Q(x)^l}.
\ee
 Finally, we show how the partial $L$-function is expressed in terms of the last expression obtained above. 
First, notice that (\ref{Q-norm}) implies
\be
|Q(x)|^2=|N_{F/K}(\xi)|^2=|N_{K/\Q}(N_{F/K}(\xi))|=|N_{F/\Q}(\xi)|=N_{F/\Q}(\,(\xi)\,),
\ee
for $\xi=\sum x_im_i$, with $x=(x_1,\dots, x_n)\in {\bf \Lambda}+{\bf u}$. Here $N_{F/\Q}(\,(\xi)\,)$ denotes the ideal norm of the principal ideal $(\xi)$. 

Second,  from (\ref{Fb+r}) we have  that as $x$ runs though $({\bf \Lambda}+{\bf u})/\mathfrak{V}_{\f}$, 
 $\xi$ runs through $(\f\b^{-1}+r)/V_\f$. These two observations, together with (\ref{Partial-L-V}), imply
\begin{align}
[U_\f:V_\f]\mathcal L(\b,r,s) 
 &=\sum_{\xi\in (\f \b^{-1}+r)/V_\f}\overline{N_{F/K}(\xi)}^kN_{F/K}(\xi)^{-l}N_{F/\Q}(\,(\xi)\,)^{-s}\notag\\
&=\sum_{x\in({\bf \Lambda}+{\bf u})/\mathfrak{V}_{\f}}\overline{Q(x)}^k Q(x)^{-l} |Q(x)|^{-2s}\notag\\
&=\sum_{x\in({\bf \Lambda}+{\bf u})/\mathfrak{V}_{\f}}\frac{\Omega^k_s(x,M)}{Q(x)^{l}},
\end{align}
and the identity follows.
\end{proof}

As a simple corollary of the above theorem we can parametrize the Hecke $L$-function  $L(s,\chi)$  defined  in (\ref{Main-L-function}) by using Proposition \ref{fullpartial1} thus proving the Theorem \ref{mainthm-intro}. 

\begin{corollary}
\label{parametrization}
Let $\mathfrak{b}$ be an integral ideal of $F$ prime to $\f$. Let $\mathcal{L}_{\f}(\chi,\mathfrak{b};s)$ be the partial Hecke $L$-function as defined \eqref{Partial-Main-L-function}.  Let $M$, ${\bf u}$  and $P$ be as in the beginning of Section \ref{the cycle-section},  with $r=1$. Then, we have
\begin{equation}
-\det(M)\,((l-1)!)^n\,[U_\f:V_{\f}]\,\mathcal{L}_{\f}(\chi,\mathfrak{b};s)
    =
        \frac{\chi(\b)}{N_{F/\Q}(\b)^s}\,
                 \Psi_s(\E[\b,M])(P^{l-1},{\bf u},M),
\end{equation}
for $\mathrm{Re}(s)>1+\frac{k}{2}$. 
\end{corollary}

\begin{proof}
This follows immediately from Proposition \ref{fullpartial1} and Theorem \ref{maintheorem} above.
\end{proof}


\section{Proof of Lemma \ref{lemma6}}

\label{proof-lemma}

\subsection{Strategy of the proof}\label{strategy}

We now provide the proof of the key lemma.  The method is inspired by \cite{Scz}  with some additional  novelties that make the proof more explicit. We start by outlining  Sczech's proof adjusted to our setting and then explain the difference in the approach we are going to take.

For the rest of this article we are going to fix an ideal $\mathfrak{b}$, the matrix $M$, an $x\in {\bf \Lambda+u}$,  a positive integer $l$, and $P$ will always denote the polynomial defined in (\ref{P and Q}). With this convention, sometimes we will simplify the notation  by omiting the pair $(P^{l-1}, x)$ from the cocycle $\psi(\cdot)(P^{l-1},x)$, and the letters $\b, M$ from the cycles $\E[\b,M]$ and $\E[\b,M,V]$.

 We will work with the units $\epsilon_i$ rather than the matrices $A_i$. Using the bar notation we have the relation
\be
U_{\pi}:=[A_{\pi(1)}|\dots|A_{\pi(n-1)}]=M[\delta(\epsilon_{\pi(1)})|\dots|\delta(\epsilon_{\pi(n-1)})]M^{-1}.
\ee
for a permutation $\pi \in S_{n-1}$.


As in Sczech, we introduce the following construction: Consider the subgroup of diagonal matrices in $GL_n(\C)$ with determinant 1, and consider the logarithm map $\ell$ of this subgroup into $\R^n$ given by 
\begin{equation}\label{log}
\textbf{l}:\mathrm{diag}(x_1,\dots,x_n)\mapsto \big(\,2\log |x_{1}|,\dots,2\log|x_n|\,\big).
\end{equation}
The image of this subgroup under this map is a hypersurface $\textbf{H}\subset \R^n$ defined by $\sum y_i=0$. Moreover, this map allows Sczech to make an identification 
\be\label{identif}
(\varrho(\eta_1),\dots,\varrho( \eta_n))
\quad  \longleftrightarrow \quad 
[\,\textbf{l}(\delta(\eta_1)),\dots, \textbf{l}(\delta( \eta_n))\,]
\ee
of an element $(\varrho(\eta_1),\dots,\varrho( \eta_n))\in GL_n(K)^n$, for units $\eta_j\in \mathcal{O}_F^{\times}$, with the oriented simplex $[\,\textbf{l}(\delta(\eta_1)),\dots,\textbf{l}(\delta( \eta_n))\,]$ in $H$  determined by the vertices $v_j=\textbf{l}(\delta(\eta_j))$. The  orientation of this simplex is defined by the sign of 
$\det(e,v_2-v_1,\dots,v_n-v_{n-1})$, where $e=(1,\dots, 1)^T$ ( cf.  \cite{Scz} Equation (26)). 

Now let
\begin{equation}\label{hyper}
\textbf{H}_j=\{y\in \textbf{H}: y_i<0<y_j \text{ for all }i\neq j\},\quad j=1,\dots,n.
\end{equation} 
By the proof of the Dirichlet Unit Theorem, as in  
\cite{Lan}, we can find $n$ units $\theta_j$ in $V_\f$ such that 
$
\textbf{l}(\delta(\theta_j))\in \textbf{H}_j
$
for every $j$. In principle, these units belong to $\mathcal{O}_F^{\times}$, but since $V_\f$ has finite index in $\mathcal{O}_F^{\times}$,
we can replace $\theta_j$ by a positive power of it to guarantee that $\theta_j\in V_\f$. Moreover,  the simplex with vertices $\textbf{l}(\delta(\theta_j))$ is positively oriented for all $j$ (cf. \cite{Scz} Section 3.2, page 596). 

 Sczech uses this construction to define the element $\mathfrak{S}_N\in GL_n(K)^n$ as
   $$\mathfrak{S}_N:=(\varrho(\theta_1^N),\dots, \varrho(\theta_n^N)),\quad N\geq 1,$$
and then shows that
\[
\lim_{N\to \infty}\psi(\mathfrak{S}_N)(P^{l-1},x)=\det(M)\frac{((l-1)!)^n}{Q(x)^l}.
\]
This is the right-hand side of  (\ref{deteq}). In order to relate this to the left-hand side of (\ref{deteq}), Sczech's idea is to consider the element $\mathfrak{I}_N$ such that $\lim_{N\to \infty} \psi(\mathfrak{I}_N)$ is equal to the left-hand side of (\ref{deteq}). More precisely, this element is the sum of all the simplices $\varrho(\eta)U_{\pi}$ for $\eta\in V_{\f}$ and $\pi \in S_{n-1}$ having non-trivial intersection with $\mathfrak{S}_N$ via the identification (\ref{identif}). The rest of the proof consists in showing that $\lim_{N\to \infty}\psi(\mathfrak{S}_N)=\lim_{N\to \infty} \psi(\mathfrak{I}_N)$, which is done by comparing the 
homology classes of $\mathfrak{S}_N$ and $\mathfrak{I}_N$.  

Therefore the comparison of the homology classes of $\mathfrak{S}_N$ and $\mathfrak{I}_N$  is the main component of Sczech's proof. However, it is more natural to compare the homology of cycles than to compare the homology of  just a  collection of simplices which do not necessarily  constitute a cycle.  
Based on this observation, we construct  the following two cycles 
 containing $\mathfrak{S}_N$ and $\mathfrak{I}_N$, and compare their homology instead.

Returning to our setting, let $V$ be the group generated by the units 
\be\label{Units-V}
\epsilon_i:=\theta_{i+1}\theta_i^{-1},\quad i=1,\dots, n-1.
\ee
 Since   $V$ is a subgroup of finite index in 
$V_{\mathfrak{f}}$ we let $\beta:=[V_{\f}:V]$. Let $\E[V]:=\E[\mathfrak{b},M,V]$ and $\rho$ be  defined as in (\ref{cycle-V})
for the units $\epsilon_1,\dots, \epsilon_{n-1}$.
From now on $N$ will denote  a positive integer divisible by $\beta$;
this guarantees that $\theta^N_i\in V$, for $i=1,\dots, n$. We consider the  following two cycles
\begin{equation}\label{VN}
\E_N[V]:=\rho\sum_{\pi\in S_{n-1}} \text{sign}(\pi)[\varrho(\epsilon^N_{\pi(n-1)})|\dots|\varrho(\epsilon^N_{\pi(1)})]
\end{equation}
and 
\be\label{E-star}
\E^*_N[V]:=\sum_{\substack{0\leq k_i\leq N-1\\ i=1,\dots, n-1}} 
           \varrho(\,\epsilon_1^{k_1}\dots \epsilon_{n-1}^{k_{n-1}}\,)\,\E[V].
\ee

 Geometrically, using   the identification (\ref{identif}), these cycles can be interpreted 
 as two different triangulations, the latter finer that the former, of the fundamental domain in $\textbf{H}$ determined 
by the  vertices  
$\textbf{l}(\delta(1)),\textbf{l}(\delta(\epsilon_1^N))$,$\dots$, 
$\textbf{l}(\delta(\epsilon_1^N\cdots\epsilon_{n-1}^N))$.

Notice that   the simplex $\mathfrak{S}_N$ is one of the summands of $\varrho(\theta_1)^N\E_N[V]$---indeed, 
$\mathfrak{S}_N=\varrho(\theta_1^N)\big[\,\varrho(\epsilon^N_{1})|\dots|\varrho(\epsilon^N_{n-1})\,\big]$--- and as such we show in Proposition \ref{Main-prop} that
\be\label{limit-2}
\lim_{N\to \infty}\psi\big(\,\varrho(\theta_1^N)\, \E_N[V]\,\big)(P^{l-1},x)
=\det(M)\frac{((l-1)!)^n}{Q(x)^l}.
\ee
Since $\mathfrak{I}_N$ is contained in the cycle $\varrho(\theta_1^N)\,\E^*_N[V]$, we also show in the same proposition that
\be\label{limit-3}
\lim_{N\to \infty} \psi\big(\,\varrho(\theta_1^N)\,\E^*_N[V]\,\big)(P^{l-1},x)
=\sum_{A \in {\mathfrak{V}}}\psi\big(\,A\, \E[V]\, \big)(P^{l-1},x).
\ee
However, the infinite sum in (\ref{limit-3})  coincides with the left-hand side of   (\ref{deteq})  since,  as it will be shown in Lemma \ref{independence} below, the value of the infinite sum
is independent of the choice of a subgroup $V$ of finite index in $V_\f$.

Therefore, it remains to compare 
\[
\psi(\varrho(\theta_1^N)\E_N[V])\quad \text{ and}\quad \psi(\varrho(\theta_1^N)\E_N^*[V]),
\]
as $N\to \infty$.
According to Lemma 5 of \cite[p.592]{Scz} these two cycles  $\E_N[V]$ and $\E^*_N[V]$ are homologous in $H_{n-1}(\mathfrak{V},\Z)$, but not necessarily in $H_{n-1}(\mathfrak{V},\Z[\mathfrak{V}])$; here $\mathfrak{V}=\varrho(V)$. Therefore $\psi(\varrho(\theta_1^N)\E_N[V])$ is not necessarily equal to $\psi(\varrho(\theta_1^N)\E^*_N[V])$. Nevertheless, we  show in Proposition \ref{Main-prop}  that
\be\label{limit-1}
\lim_{N\to \infty} \psi\big(\,\varrho(\theta_1^N)\,\E_N[V]\,\big)(P^{l-1},x)
=\lim_{N\to \infty} \psi\big(\,\varrho(\theta_1^N)\,\E^*_N[V]\,\big)(P^{l-1},x).
\ee
Putting together  (\ref{limit-2}), (\ref{limit-3}) and (\ref{limit-1}), we obtain (\ref{deteq}). Our goal for  the rest of the article is then to prove these three limits.

\subsection{Invariance of (\ref{deteq})}\label{Invariance-of-sum}
As was mentioned above, we start by showing that (\ref{deteq}) is independent of the choice of a  subgroup $V$ of finite index in $V_\f$. 
\begin{lemma}
\label{independence}
The infinite sum in equation \eqref{deteq}
\be\label{Infinite-sum}
\sum_{A \in {\mathfrak{V}_\f}}\psi\big(\,A\, \E[\b,M]\,\big)(P^{l-1},x)
\ee
is independent of the set of generators of the group $V_\f$. Moreover, the value of sum does not change if we replace $V_\f$ by any subgroup $V$ of  finite index in $V_\f$, i.e.,
\be\label{sum-invariance}
\sum_{A \in {\mathfrak{V}_\f}}\psi\big(\,A\, \E[\b,M]\, \big)(P^{l-1},x)
   =\sum_{A \in \mathfrak{V}}\psi\big(\,A\, \E[\b,M,V]\,\big )(P^{l-1},x),
\ee
where $\mathfrak{V}=\varrho(V)$.
\end{lemma}
\begin{proof}
This follows from the fact that $H_{n-1}(\mathfrak{V}_\f,\mathbb{Z})\simeq \Z$ is cyclic and generated by $[\E]$, the class of $\E:=\E[\mathfrak{b},M]$ in $H_{n-1}(\mathfrak{V}_\f,\mathbb{Z})$ (c.f. Lemma 5 of \cite[p.592]{Scz}). Indeed, let $V$ be a subgroup of $V_\f$ of finite index.  Fix a set of generators $\epsilon_1',\dots, \epsilon'_{n-1}$  of $V$ and let $\tilde{\E}$ denote $\E[\mathfrak{b},M,V]$.
Thus $[\tilde{\E}]=[V_\f:V][\E]$ in $H_{n-1}(\mathfrak{V}_\f,\mathbb{Z})$ and therefore
\begin{align*}
\sum_{A \in \mathfrak{V}}\psi\big(\,A\, \tilde{\E}\,\big )(P^{l-1},x)
=&\frac{1}{[V_\f:V]}\sum_{A \in \mathfrak{V}_\f}\psi\big(\,A\, \tilde{\E}\,\big )(P^{l-1},x)\\
=&\frac{1}{[V_\f:V]}\sum_{A \in \mathfrak{V}_\f}\psi\big(\,A\, \big([V_\f:V]\E\big)\,\big )(P^{l-1},x)\\
=&\frac{1}{[V_\f:V]}\sum_{A \in \mathfrak{V}_\f}[V_\f:V]\,\psi\big(\,A\, \E\,\big)(P^{l-1},x)\\
=&\sum_{A \in \mathfrak{V}_\f}\psi\big(\,A\, \E\,\big )(P^{l-1},x).\\
\end{align*}
\end{proof}

\subsection{Comparing the homology of $\E_N[V]$ and $\E_N^*[V]$}\label{Comparing homology}
The cycles $\E_N[V]$ and $\E_N^*[V]$   are homologous in $H_{n-1}(\mathfrak{V},\Z)$, but not necessarily in $H_{n-1}(\mathfrak{V},\Z[\mathfrak{V}])$. Therefore  $\psi(\varrho(\theta_1^N)\E_N[V])$ is not necessarily equal to $\psi(\varrho(\theta_1^N)\E^*_N[V])$. In this section we will estimate the difference 
\be\label{estim-diff}
\psi(\varrho(\theta_1^N)\E_N[V])-\psi(\varrho(\theta_1^N)\E^*_N[V])
\ee by
 comparing the homology of 
$\E_N[V]$ and $\E_N^*[V]$ in $H_{n-1}(\mathfrak{V},\Z[\mathfrak{V}])$. We will conclude in Lemma \ref{simplex} that
$$
\E_N[V]-\E_N^*[V]\text{\ is homologous in $H_{n-1}(\mathfrak{V},\Z[\mathfrak{V}])$ to\ } \mathfrak{r}_N, 
$$
where $\mathfrak{r}_N$ is an element in $\Z[\mathfrak{V}^{n}]$ being the 
sum of simplices laying on the faces of the  fundamental domain  in $\textbf{H}$ determined by the vertices  
 $$\textbf{l}(\delta(1)),\textbf{l}(\delta(\epsilon_{1}^N)),\dots, \textbf{l}(\delta(\epsilon_{1}^N\cdots \epsilon_{n-1}^N))$$
 via the correspondence (\ref{identif}). 
The precise description of $\mathfrak{r}_N$ is given in Lemma \ref{simplex}. Then, in Lemma \ref{boundary}, we will estimate the value of the cocycle $\psi$ at $\varrho(\theta_1^N)\,\mathfrak{r}_N$ and conclude that 
$
\lim_{N\to \infty} \psi\big(\,\varrho(\theta_1^N)\,\mathfrak{r}_N\,\big)(P^{l-1},x)=0,
$
therefore showing (\ref{limit-1}).

The existence of the element $\mathfrak{r}_N$ is based on an inductive 
argument involving the $k$-dimensional faces ( $k=0$, 1, $\dots$ , $n-1$) of the cycles $\E_N[V]$ and $\E_N^*[V]$, for this reason we will need to introduce
 the following notation which  allows us to simplify 
our computations significantly. In particular, it will give us a convenient way to express the faces of $\E_N[V]$ and $\E_N^*[V]$. 

Let $\epsilon_1,\dots, \epsilon_{n-1}$ be the 
units (\ref{Units-V}). For a subset  
$S=\{\epsilon_{i_1},\dots, \epsilon_{i_k}\}$ of $\epsilon_1\dots, \epsilon_{n-1}$, where 
$1\leq i_1\leq \dots \leq i_k\leq n-1$, we let
\be
\E[S]
:=\rho\sum_{\pi\in S_k} 
\text{sign}(\pi)\big[\varrho\big(\epsilon_{i_{\pi(1)}}\big)\,\big|\dots\,\big|\varrho\big(\epsilon_{i_{\pi(k)}}\big)\,\big],
\ee
where $\rho$ is as in Definition \ref{cycle}, i.e., $
\rho=(-1)^{n-1}\text{sign}(R_{F/K})$
where $R_{F/K}=R(\epsilon_1,\dots, \epsilon_{n-1})$. Additionally, we let $S_N=\{\epsilon_{i_1}^N,\dots, \epsilon_{i_k}^N\}$ and also define
\be
\E_N[S]:=\E[S_N] \quad \text{and} \quad \E_N^*[S]:=\sum_{\substack{0\leq m_i\leq N-1\\ i=1,\,\dots,\,k}} 
           \varrho(\epsilon_{i_1}^{m_1}\dots \epsilon_{i_{k}}^{m_{k}})\,\E[S].
\ee 
When $S$ is the full set $\{\epsilon_1,\dots, \epsilon_{n-1}\}$ we have $\E_N[S]=\E_N[V]$ and $\E_N^*[S]=\E_N^*[V]$. For each $S$, $\E_N[S]$ and $\E_N^*[S]$  correspond to a $k$-dimensional face of $\E_N[V]$ and $\E_N^*[V]$, respectively, via the identification (\ref{identif}).

In order to simplify the notation further, we denote by $S(\epsilon_{i_k})$ the index $k$; thus $S(\epsilon)$ is the position of the element $\epsilon$ in the ordered set $S$. This allows us to express
the boundary map $\partial$ as
\begin{align*}
\partial(\varrho(\epsilon_{i_1}),\dots, \varrho(\epsilon_{i_k}))
&=\sum_{j=1}^k(-1)^j(\varrho(\epsilon_{i_1}),\dots,\widehat{\varrho(\epsilon_{i_{j}})},\dots,\varrho(\epsilon_{i_k}))\\
&=\sum_{\epsilon\in S}(-1)^{S(\epsilon)}(\varrho(\epsilon_{i_1}),\dots,\widehat{\varrho(\epsilon)},\dots,\varrho(\epsilon_{i_k})),
\end{align*}
where $\widehat{\varrho(\epsilon)}$ means the omission of the element $\varrho(\epsilon)$.

Also, if $\epsilon$ and $\epsilon'$ are elements of $S$, we will denote by $S_{\epsilon}$ and $S_{\epsilon,\epsilon'}$ the sets $S\setminus \{\epsilon\}$ and $S\setminus \{\epsilon,\epsilon'\}$, respectively. Moreover, we will denote by $\mathfrak{S}$, $\mathfrak{S}_{\epsilon}$ and $\mathfrak{S}_{\epsilon,\epsilon'}$, respectively, the subgroups  of $\mathfrak{V}$ generated by the sets $\varrho(S)$, $\varrho(S_{\epsilon})$ and $\varrho(S_{\epsilon,\epsilon'})$.

With this notation, we have
\begin{equation}\label{boundary-1}
\partial\big(\,\E_N[S]\,\big)=
\sum_{\epsilon\in S}\, (-1)^{S(\epsilon)}\,(1-\varrho(\epsilon^N))\, \E_N[S_{\epsilon}],
\end{equation}
i.e., the boundary map evaluated at $\E_N[S]$ corresponds  exactly to the sum of all of its faces, $ \E_N[S_{\epsilon}] $ and $\varrho(\epsilon^N) \E_N[S_{\epsilon}]$ for $\epsilon\in S$, with the appropriate orientations. Subsequently we obtain
\begin{equation}\label{boundary-2}
\partial\big(\E^*_N[S]\,\big)=
\sum_{\epsilon\in S} (-1)^{S(\epsilon)}(1-\varrho(\epsilon^N)) \E^*_N[S_\epsilon].
\end{equation}

Finally, for the proof of the next lemma we need to introduce the map $h:\Z[\mathfrak{V}^{i}]\to \Z[\mathfrak{V}^{i+1}]$ ($i\geq 1$) ( c.f. \cite[\S2]{AW}),  such that any tuple $(g_1,\dots, g_i)\in \mathfrak{V}^{i}$ is sent to $(1,g_1,\dots, g_i)$. This map satisfies 
 \begin{equation}\label{partial-h}
 \partial h+h\partial=1.
 \end{equation}

\begin{lemma}\label{simplex}
Let $N\geq 1$ and  $\epsilon_1,\dots, \epsilon_{n-1}$ be the units defined in \eqref{Units-V}. Then there exists an element $\mathfrak{r}_N$ in $\Z[\mathfrak{V}^{n}]$
such that
\[
\E_N[V]\ \sim\ \E^*_N[V]+ \mathfrak{r}_N
\]
in $H_{n-1}(\mathfrak{V},\Z[\mathfrak{V}])$, i.e., $\E_N[V]= \E^*_N[V]+ \mathfrak{r}_N$ $\normalfont \text{mod}$ $\partial(\Z[\mathfrak{V}^{n+1}])$.

Moreover,  $\mathfrak{r}_N$  is generated by  $O(N^{n-2})$ elements---the growth depending only on $n$---each one  of the form
\be\label{elem-gen}
(\varrho(\eta_1^{(i)}),\dots,\varrho(\eta_{n}^{(i)}))
\quad \text{or } \quad 
\varrho(\epsilon_i^N)(\varrho(\eta_1^{(i)}),\dots,\varrho(\eta_{n}^{(i)})) 
\ee
for some $1\leq i \leq n-1$, where the units $\eta_1^{(i)},\dots,\eta_n^{(i)}$ are of the form
\be\label{special-unit}
\eta_j^{(i)}
=
\epsilon_1^{s_{1j}}
       \cdots 
          \widehat{\epsilon_i}^{s_{ij}} 
            \cdots 
            \epsilon_{n-1}^{s_{n-1j}} 
             \hspace{25pt} (1\leq j \leq n)
\ee
for integers $0\leq s_{1j},\dots, s_{n-1 j}\leq N$. Here  $ \widehat{\epsilon_i}$ means the omission of the unit $\epsilon_i$. 
\end{lemma}

\begin{proof}
 
The  proof will be a  particular case, namely $S=\{\epsilon_1,\dots, \epsilon_{n-1}\}$, of the following more general result. For any subgroup  $S=\{\epsilon_{i_1},\dots, \epsilon_{i_k}\}$ of $\{\epsilon_1,\dots, \epsilon_{n-1}\}$, with $1\leq k \leq n-1$, we will show the following:
\begin{enumerate}
\item There exist elements $\mathfrak{r}_N(S)\in \Z[\mathfrak{V}^{k+1}]$ and
$\beta_N(S)\in \Z[\mathfrak{V}^{k+2}]$ such that 
\be
\E_N[S]-\E^*_N[S]
=\mathfrak{r}_N(S)+\partial(\beta_N(S)),
\ee
satisfying  
\[
\mathfrak{r}_N(S)=\sum_{\epsilon\in S}\,(-1)^{S(\epsilon)}(1-\varrho(\epsilon)^N)\,\beta_N(S_{\epsilon}).
\]
 \item More specifically, $\beta_N(S)$ belongs  to the ring $\Z[\mathfrak{S}^{k+2}]$ and the canonical representation of $\beta_N(S)$ in this last ring is of the form
\begin{equation}\label{canonical}
\sum_{\overline{j}=(j_1,\dots, j_{k+2})\in J}m_{\overline{j}}\,(g_{j_1},\dots, g_{j_{k+2}}),
\end{equation}
where $m_{\overline{j}}\in \Z$ and each  $g_j$ is of the form $\varrho(\epsilon_{i_1}^{s_1}\cdots \epsilon_{i_k}^{s_k})$ with $0\leq s_1, \dots, s_k\leq  N$.  
\item Moreover,
the number of canonical generators appearing in the representation (\ref{canonical}) of $\beta_N(S)$ is
 of the  order $O(N^k)$, i.e., $|J|=O(N^k)$.
 \end{enumerate}

The construction is done inductively on the number of elements of the subset $S$. If $|S|=1$, then 
\be
\partial(\E_N[\epsilon_1]-\E^*_N[\epsilon_1])=0,
\ee
so $\mathfrak{r}(\epsilon_1)=0$, $\beta_N(\emptyset)=0$ and 
$\beta_N(\epsilon_1)
  =h\big(\,\E_N[\epsilon_1]-\E^*_N[\epsilon_1]\,\big)
  \in \Z[\varrho( \langle \epsilon_1 \rangle)^3 ]$.
Moreover, since $\E^*_N[\epsilon_1]$ has $N$ elements, 
then the number of generators of $\beta_N(\epsilon_1)$
is of the order $O(N)$, as $N\to \infty$.

Suppose now that the result is true for any subset of order less than $n$ and let $S=\{\epsilon_1\dots, \epsilon_n\}$. From Equations (\ref{boundary-1}) and (\ref{boundary-2}) we have
\begin{equation}\label{tomata}
\partial\big(\,\E_N[S]-\E^*_N[S]\,\big)=
\sum_{\epsilon\in S} (-1)^{S(\epsilon)}(1-\varrho(\epsilon^N)) \left(\,\E_N[S_{\epsilon}]
-\E^*_N[S_\epsilon]\,\right).
\end{equation}
Since $|S_\epsilon|=n-1$
 there exist  elements 
 $\beta_N(S_{\epsilon})
 \in \Z[\mathfrak{S}_{\epsilon}^{n+1}]$ such that
\be\label{fuck}
\E_N[S_{\epsilon}]
        -\E^*_N[S_{\epsilon}]
=\sum_{\epsilon'\in S_{\epsilon}}^n(-1)^{S_{\epsilon}(\epsilon')}\,
     (1-\varrho(\epsilon'^{N}))
     \,\beta_N(S_{\epsilon,\epsilon'})
           +\partial(\beta_N(S_{\epsilon})),
\ee
where $\beta_N(S_{\epsilon,\epsilon'})\in \Z[\mathfrak{S}_{\epsilon,\epsilon'}^{n}]$. 
Replacing $\E_N[S_{\epsilon}]
-\E^*_N[S_\epsilon]$  in (\ref{tomata}) with the expression on the right-hand side of (\ref{fuck})
 we get
\be
\sum_{\epsilon \in S}\sum_{\epsilon'\in S_{\epsilon}}(-1)^{S(\epsilon)+S_{\epsilon}(\epsilon')}
    \,(1-\varrho(\epsilon^N))\,
   (1-\varrho(\epsilon'^N))
   \,\beta_N(S_{\epsilon,\epsilon'})\ 
+\ \sum_{\epsilon\in S}\partial\left(\,(-1)^{S(\epsilon)}(1-\varrho(\epsilon^N)) \beta_N(S_{\epsilon})\right).
\ee
But the double sum  is equal to zero since for every pair of elements $\epsilon\neq \epsilon'$ of $S$ the term $(1-\varrho(\epsilon^N))(1-\varrho(\epsilon'^N))\beta_N(S_{\epsilon,\epsilon'})$ appears twice in the  sum,  but with opposite signs. Thus 
\be
\partial\left(\,\E_N[S]-\E^*_N[S]\,-\sum_{\epsilon\in S}(-1)^{S(\epsilon)}(1-\varrho(\epsilon^N))\beta_N(S_{\epsilon}) \right)=0.
\ee
Therefore we define 
$\beta_N(S)
\in \Z[\mathfrak{S}^{n+2}]$ as 
\be\label{element}
\beta_N(S):=h\left( \E_N[S]-\E^*_N[S]\,-\sum_{\epsilon\in S}(-1)^{S(\epsilon)}(1-\varrho(\epsilon^N))\beta_N(S_{\epsilon}) \right)
\ee
and the first part of the lemma is proven by (\ref{partial-h}). For the second part, observe that the number of generators of $\beta_N(S)$ is of the order $O(N^{n})$,
 since $\E^*_N[S]$ 
 has $N^{n}$ generators, and 
 $\beta_N(S_{\epsilon})$ 
 has $O(N^{n})$ generators by the induction hypothesis.
\end{proof}

\subsection{Estimating \eqref{estim-diff}}\label{Estimating difference}
 For a fixed $x\in {\bf \Lambda+u}$, 
$P$ as  in (\ref{P and Q}) and  a positive integer $l$, our goal  now is to show that   
$$\lim_{N\to \infty} \psi\left(\,\varrho(\theta_1^N)\, \mathfrak{r}_N\,\right)(P^{l-1},x)=0.$$
In order to do this, we will first estimate the values of $\psi$ at every one of the 
simplices contained in $\varrho(\theta_1^N) \mathfrak{r}_N$. 
According to Lemma \ref{simplex}, $\varrho(\theta_1^N)\, \mathfrak{r}_N$ 
is generated by $O(N^{n-2})$ elements $(A_1,\dots, A_n)\in GL_n(K)^n$ 
which are
either of the form
\begin{equation}\label{simplexx-1}
\big(\,
    \varrho\big(\,\theta_1 ^N\,\eta_1^{(i)}\,\big)
     \, ,\dots,\,
      \varrho\big(\theta_1 ^N\,\eta_n^{(i)}\big)
\,\big)
=M\,\big(\,
    \delta\big(\theta_1 ^N\,\eta_1^{(i)}\big)
          \, ,\dots,\,
              \delta\big(\theta_1 ^N\,\eta_n^{(i)}\big)\,\big)
                \,M^{-1}
\end{equation}
or
\begin{equation}\label{simplexx-2}
\big(\,
 \varrho\big(\theta_1 ^N\,\epsilon _i^N \eta_1^{(i)}\big)\,
   ,\dots,\,
   \varrho\big(\theta_1 ^N\,\epsilon_i^N\,\eta_n^{(i)}\big)
 \,\big)
=M\,
\big(\,
 \delta\big(\theta_1 ^N\,\epsilon _i^N \eta_1^{(i)}\big)\,
   ,\dots,\,
   \delta\big(\theta_1 ^N\,\epsilon_i^N\,\eta_n^{(i)}\big)
 \,\big)
 \,M^{-1},
\end{equation}
for some $1\leq i \leq n-1$. Here the  units $\eta_1^{(i)},\dots, \eta_{n}^{(i)}$ are of the form (\ref{special-unit}).
We will estimate the value of $\psi$ at both (\ref{simplexx-1}) and (\ref{simplexx-1}) in  Lemma \ref{boundary}.

Additionally, we will need to  make the following further 
assumptions on the units $\theta_1,\dots, \theta_n$ for the 
proofs of  Lemma \ref{boundary} and Proposition \ref{Main-prop}. 
For every  $1\leq i\leq n-1$ we will assume that the norm of  
$\textbf{l}(\delta(\theta_{i+1}))$ is much larger with respect to the norms 
of each one of the vectors  $\textbf{l}(\delta(\theta_1)),\dots, \textbf{l}(\delta(\theta_i)) $. 
To be more specific, for all  $1\leq i\leq n-1$ we may assume, 
by replacing $\theta_{i+1}$ for $\theta_{i+1}^s$ for a large integer $s$,  that
\be\label{extra-assump}
\textbf{l}(\delta(\epsilon_i)),\textbf{l}(\delta(\epsilon_{i-1}\epsilon_i)),\dots,\textbf{l}(\delta(\epsilon_{1}\cdots\epsilon_i)) \in \textbf{H}_{i+1},
\ee 
\be\label{extra-assump-1.5}
\textbf{l}(\delta(\theta_1\,\epsilon_i)),\textbf{l}(\delta(\theta_1\,\epsilon_{i-1}\epsilon_i)),\dots,\textbf{l}(\delta(\theta_1\epsilon_{1}\cdots\epsilon_i)) \in \textbf{H}_{i+1},
\ee
and
\be\label{extra-assump-2}
|\rho_1(\theta_1\,\epsilon_i)|\leq |\rho_1(\theta_i)|. 
\ee
With these assumptions we are now ready to prove the following result which establishes the Equation (\ref{limit-1}).
\begin{lemma}\label{boundary}
Fix $x\in {\bf \Lambda+u}$.  Let  $P$  be as  in \eqref{P and Q} and $l$ be a positive integer. Let $\mathfrak{r}_N$ be as in Lemma \ref{simplex}.  Then
\be\label{hom-limit}
\lim_{N\to \infty} \psi\left(\,\varrho(\theta_1^N)\, \mathfrak{r}_N\,\right)(P^{l-1},x)=0.
\ee
\end{lemma}
\begin{proof}
Recall that $\varrho(\theta_1^N)\,\mathfrak{r}_N$ has $O(N^{n-2})$
generators $(A_1,\dots, A_n)\in GL_n(K)^n$  which are either of the from (\ref{simplexx-1}) or (\ref{simplexx-2}) for some $1\leq i \leq n-1$.  
The crucial step in  the proof is to obtain the following estimate
\begin{equation}\label{estimation}
\psi(A_1,\dots, A_n)(P^{l-1},x\,)=O(t^N) \quad  (\, N\to \infty\, ),
\end{equation}
where $t=\max_{1\leq k\neq j\leq n}\big\{|\rho_k(\theta_j)|\big\}<1.$   The constant in the $O$ notation depends only on $x$ and $M$. From (\ref{estimation})  we can deduce that
\[
 \psi\left(\,\varrho(\theta_1^N)\, \mathfrak{r}_N\,\right)(P^{l-1},x)
 =O\big(\,N^{n-2}\,t^N\,\big) \quad (\, N\to \infty\,),
\]
which proves (\ref{hom-limit}). 

Thus, it will be enough to show (\ref{estimation}). For this, fix an $1\leq i \leq n-1$. Let $\eta_1^{(i)},\dots, \eta_n^{(i)}$ be the units corresponding to $(A_1,\dots, A_n)$ in either case (\ref{simplexx-1}) or (\ref{simplexx-2}).  Here each $\eta_j^{(i)}$ of the form $\epsilon_1^{s_{1j}}
             \cdots 
             \widehat{\epsilon}_i^{\,s_{ij}} 
             \cdots 
             \epsilon_{n-1}^{s_{n-1j}}$  as  in (\ref{special-unit}). 
           By  (\ref{Action-of-A-on-f})  we have
$
\psi(A_1,\dots, A_n)(P^{l-1},x\,)=\det(M)\,\psi(\mathfrak{B})(M^TP^{l-1},xM\,),
 $ 
where  $\mathfrak{B}$ is the $n$-tuple of matrices $(B_1,\dots, B_n)$ determined by
 \be\label{simplexe-1}
             \mathfrak{B}=\big(\,\delta(\theta_1^N\,\eta_1^{(i)})\,,\dots,\, \delta(\theta_1^N\,\eta_n^{(i)})\,\big)M^{-1}\,,
             \ee
 in the case  (\ref{simplexx-1}),            or
             \be\label{simplexe-2}
             \mathfrak{B}=\big(\,\delta(\theta_1^N\,\epsilon_i^N\,\eta_1^{(i)})
             \,,\dots,\, \delta(\theta_1^N\,\epsilon_i^N\,\eta_n^{(i)})\,\big)M^{-1}\,,
             \ee 
             in the case   (\ref{simplexx-2}). Therefore it is enough to show, for (\ref{simplexe-1}) and (\ref{simplexe-2}), that
\be\label{estimation-other}
\psi(\mathfrak{B})(M^TP^{l-1},xM\,)=O(t^N) \quad (N\to \infty).
\ee

If we let  $(B_{1j_1},\dots,B_{nj_n})$ be the matrix from the definition of $\psi$ (cf. (\ref{Def-of-psi})), i.e., $B_{kj_k}$ is the first 
column of the matrix $B_{k}$ such that 
$\langle xM, B_{kj_k} \rangle\neq 0$, 
then $$\psi(\mathfrak{B})(M^TP^{l-1},xM\,)=f(B_{1j_1},\dots,B_{nj_n})(M^TP^{l-1},xM\,).$$ Moreover, since $f$ is homogeneous by its definition,   we can replace $(B_{1j_1},\dots,B_{nj_n})$ with any matrix $\sigma$ obtained by re-scaling each one of  the columns  $B_{kj_k}$; this allows us to work  with $f(\sigma)(M^TP^{l-1},xM\,)$ instead.  Such a choice of matrix $\sigma$ will be specified below. Thus in order to show  (\ref{estimation-other}), we will bound $f(\sigma)(M^TP^{l-1},xM\,)$ by estimating each one of the terms in the equivalent expression  (\ref{Pr}). In particular,  the key step will be to show that $\det(\sigma)=O(t^N)$, as $N\to \infty$.

This estimate on $\det(\sigma)$ will be deduced  from the fact that   
 $\epsilon_i$ is missing in the decomposition of 
 all the units $\eta_j^{(i)}$.  
In the case (\ref{simplexe-1})  this will be obtained
 by expanding the determinant $\det(\sigma)$ along 
 the $(i+1)$th row, and in the  case (\ref{simplexe-2}) 
 by expanding $\det(\sigma)$ along the first row.         
            
            We start with the following observation: Since  $\textbf{l}(\delta(\epsilon_k))\in H_{k+1}$, for $k=1,\dots, n-1$ (cf. (\ref{extra-assump})), then $\textbf{l}(\delta(\epsilon_k^{s}))\in \textbf{H}_{k+1}$ for $s\geq 0$. Thus it  follows, from the very definition of the $\textbf{H}_k$'s,  that the first and $(i+1)$th diagonal elements of 
\[\textbf{l}(\delta(\eta_j^{(i)}))
=
\textbf{l}
\big(\,
\delta(
      \epsilon_1^{s_{1j}}
             \cdots 
             \widehat{\epsilon}_i^{\,s_{ij}} 
             \cdots 
             \epsilon_{n-1}^{s_{n-1j}}
      )\,
\big)
\] 
are negative; bearing in mind that the  $s_{kj}$'s are non-negative integers. This implies that
\be\label{ine}
|\rho_1(\eta_j^{(i)})|<1 \quad \text{and}\quad |\rho_{i+1}(\eta_j^{(i)})|<1. 
\ee 

 From  (\ref{ine}), we can bound the $(i+1)$th component of each one of the diagonal matrices $\delta(\theta_1^N\,\eta_j^{(i)})$ in (\ref{simplexe-1}) as
\be\label{bound-i+1-row}
\big|\rho_{i+1}(\theta_1^{N} \,\eta_j^{(i)})\big|
\leq |\rho_{i+1}(\theta_1)^N|\leq t^N,
\ee
and   in view of (\ref{extra-assump-2}), we can also bound  the first component of each  one of the diagonal 
 matrices $\delta(\theta_1^N\,\epsilon_i^N\,\eta_j^{(i)})$ in (\ref{simplexe-2}) as
\be\label{bound-1-row}
\big|\rho_{1}(\theta_1^{N}\epsilon_i^N \eta_j^{(i)})\big|
\leq |\rho_{1}(\theta_1 \, \epsilon_i)|^N\leq |\rho_{1}(\theta_{i+1})|^N\leq t^N.
\ee

Moreover, since both $\textbf{l}\big(\delta(\theta_1^N
      \eta_j^{(i)}
      )\big)$ and $\textbf{l}\big(\delta(\theta_1^N\,\epsilon_i^N\,
      \eta_j^{(i)}
      )\big)$ lie  in  the hypersurface $\textbf{H}$ at least one of its diagonal   elements must be positive.  This implies that 
\be
\gamma_k:=\max_{1\leq r \leq n}
    \big\{\,|\rho_r(\theta_1^N\,\eta_j^{(i)})|\,\big\}>1
    \quad \text{and}\quad 
    \omega_k:=\max_{1\leq r \leq n}
    \big\{\,|\rho_r(\theta_1^N\,\epsilon_i^N\,\eta_j^{(i)})|\,\big\}> 1.
\ee

As was mentioned  previously, we can work with the re-scaled matrix $\sigma=(\sigma_1,\dots,\sigma_n)$, where  
\be\label{re-scale}
\sigma_k:=B_{kj_k}/\gamma_k\quad \text{or} \quad  \sigma_k:=B_{kj_k}/\omega_k
\ee
in the cases (\ref{simplexe-1}) and  (\ref{simplexe-2}), respectively. Thus, $$\psi(\mathfrak{B})(M^TP^{l-1},xM\,)=f(\sigma)(M^TP^{l-1},xM\,).$$ 
In order to obtain (\ref{estimation-other}) we will estimate each one of the terms of the equivalent expression (\ref{Pr}) for $f(\sigma)(M^TP^{l-1},xM\,)$. 

The purpose of the re-scaling (\ref{re-scale}) is that each component of the matrix $\sigma$  is of the form $O(1)$, where the bound depends on $M$ only. From this we conclude first  that $1/\langle xM,\sigma_k\rangle=O(1)$ as $N\to \infty$. 
Second, that the coefficients $P_r(\sigma)$ are $O(1)$ as well: Indeed for $P$ as in (\ref{P and Q}), we have $M^TP^{l-1}(x)=P^{l-1}(xM^T)=(x_1\cdots x_n)^{l-1}$, thus the coefficients $P_r(\sigma)$ in (\ref{Pr}) and (\ref{Pr-2}) are of the form (\ref{Pr-coeff}). 

Finally,  observe that the $(i+1)$th row of the matrix $\sigma $ is of 
the order $O(t^N)$ by (\ref{bound-i+1-row}) in the case (\ref{simplexe-1}),  and the first row of the matrix $\sigma $ is of 
the order $O(t^N)$ by (\ref{bound-1-row}) in the case (\ref{simplexe-2}). Thus, by expanding the determinant of $\sigma$ along the  $(i+1)$th row, and first row, respectively, we see that $\det(\sigma)=O(t^N)$, as $N\to \infty$. Hence (\ref{estimation-other}), and consequently  (\ref{estimation}),  follows.
\end{proof}
\subsection{Proof of the limits (\ref{limit-2}), (\ref{limit-3}) and (\ref{limit-1})}\label{Three limits}
Putting all the above lemmas together we finally obtain the following.
\begin{proposition}\label{Main-prop} Fix $x\in {\bf \Lambda+u}$, $P$ as  in \eqref{P and Q} and $l$ a positive integer. Then three limits \eqref{limit-2}, \eqref{limit-3} and \eqref{limit-1} hold.
\end{proposition}
\begin{proof}
By Lemma \ref{simplex} we have
 $
\E_N[V]\sim \E_N^*[V]+\mathfrak{r}_N, 
$
 and by Lemma  \ref{boundary}
 \be
\lim_{N\to \infty} \psi\left(\,\varrho(\theta_1)^N\, \mathfrak{r}_N\,\right)(P^{l-1},x)=0,
\ee
thus the  limit (\ref{limit-1}) follows.

 The  limit (\ref{limit-2}) is clear since  from the 
very definition of the cycle $\E^*_N[V]$ we have
\be
\psi\left(\,\varrho(\theta_1^N)\,\E^*_N[V]\,\right)
=\sum_{\substack{0\leq k_i\leq N-1\\ i=1,\,\dots,\,n-1}} 
           \psi\left(\varrho(\theta_1^N\epsilon_1^{k_1}\dots \epsilon_{n-1}^{k_{n-1}})\,\E[V]\,\right).
\ee

Finally, we prove the  limit (\ref{limit-3}).
For a permutation
$\pi$ in $S_{n-1}$, we will show that $\lim_{N\to\infty}\psi\big(\, \varrho(\theta_1^N)\,U_{\pi}\, \big)(P^{l-1},x)$ is equal to zero if $\pi\neq id$ and equal to  $\det(M) ((l-1)!)^n/Q(x)^l$ if $\pi=id$, therefore proving  (\ref{limit-3}).

We start by observing that  the assumption (\ref{extra-assump})   implies 
\[
\textbf{l}\big(\,\delta\big(\theta_1^N\epsilon_{\pi(1)}^N\cdots\epsilon_{\pi(i)}^N\big)\,\big)\in \textbf{H}_{k_i+1}
\]
where $k_i$ is the largest integer among $\pi(1),\dots, \pi(i)$. This implies
\be
\delta(\theta_1^N\epsilon_{\pi(1)}^N\cdots\epsilon_{\pi(i)}^N\big)
\to E_{k_i+1}, \quad (N\to \infty)
\ee
 where $E_k$ denoted the $n\times n$ matrix with a 1 in the $k$th diagonal component and zeros everywhere.
Thus modulo projective equivalence
\be
\varrho(\theta_1^N)\,U_{\pi}\to M(E_1,E_{k_1+1},\dots, E_{k_{n-1}+1})M^{-1}.
\ee
Therefore 
\begin{equation*}
\lim_{N\to\infty}\psi\big(\, \varrho(\theta_1^N)\,U_{\pi}\, \big)(P^{l-1},x)=
\psi\big(\,M \,(E_1,E_{k_1+1},\dots, E_{k_{n-1}+1})\,M^{-1}\,\big)(P^{l-1},x).
\end{equation*}
By the definition of $\psi$ the above expression equals $f(ME_{\pi})(P^{l-1},x)$, where $E_{\pi}$ is the matrix whose first column is the first column of $E_1$, and its $(i+1)$th column, $1\leq i \leq n-1$,
is the $(k_i+1)$th column of $E_{k_i+1}$. On the other hand, $f(ME_{\pi})(P^{l-1},x)=\det(M)f(E_{\pi})(M^{T}P^{l-1}, xM)$ by  (\ref{Action-of-A-on-f}),  and so it equals 
\begin{equation}\label{limit-formula}
\det(M)\,(-1)^{n(l-1)}(\partial_{\xi_1}\dots\partial_{\xi_n})^{l-1}f(E_{\pi})(xM)
\end{equation}
where $(\xi_1,\dots, \xi_n)=xM$ by (\ref{Def-of-f}). 

 If $\pi\neq id$ then (\ref{limit-formula}) is zero by (\ref{Pr}) since $\det(E_{\pi})=0$; in this case at least two of the matrices $E_{k_1+1},\dots, E_{k_{n-1}+1}$ are repeated. As for $\pi=id$, we have that $E_\pi=1_n$ and (\ref{limit-formula}) becomes $\det(M) ((l-1)!)^n/Q(x)^l$. The third limit follows.
\end{proof}

\noindent\emph{Acknowledgments.} The authors would like to thank Robert Sczech for discussions concerning his work, and the referee for a careful reading of the paper and for helpful comments. The third author also thanks Pierre Charollois for comments on a preliminary version of this paper.

\bibliographystyle{alpha}
\bibliography{Eiscocycle}
\end{document}